\documentclass[11pt, a4paper]{article}
\usepackage{filecontents}
\usepackage{inputenc} 
\usepackage[english]{babel}
\usepackage{dsfont}

\usepackage[totalheight=24 true cm, totalwidth=17 true cm]{geometry}
\usepackage{enumitem}

\setlist[itemize]{noitemsep}
\usepackage{bbding}
\usepackage{amsmath,multirow}
\usepackage{amsthm}
\usepackage{amssymb}
\usepackage{mathrsfs}
\usepackage{mathtools}
\usepackage{ stmaryrd }
\usepackage{url}
\usepackage{wrapfig}
\usepackage{starfont}
\usepackage{pifont}
\usepackage{eurosym}
\usepackage{subcaption}

\usepackage{imakeidx}
\makeindex[]

\usepackage{multicol}
\usepackage{relsize}
\usepackage{float}
\usepackage{tikz,pgf}
\usepackage{tikz-cd}
\usepackage{wasysym}
\usepackage{graphicx}
\usepackage{fancyhdr}
\usepackage{verbatim}

\usepackage{pgfplots}
\pgfplotsset{compat=1.15}
\usepackage{mathrsfs}
\usetikzlibrary{arrows}

\usepackage[all,dvips,knot,web,arc,curve,color,frame]{xy}
\usepackage[all,knot,arc,color,web]{xy}
\xyoption{arc}
\xyoption{web}
\xyoption{curve}

\numberwithin{equation}{section}

\theoremstyle{plain}
\newtheorem{theorem}{Theorem}[section]
\newtheorem{proposition}[theorem]{Proposition}

\newtheorem{lemma}[theorem]{Lemma}
\newtheorem{notation}[theorem]{Notation}

\theoremstyle{definition}
\newtheorem{definition}[theorem]{Definition}

\newtheorem{remark}[theorem]{Remark}

\newtheorem{example}[theorem]{Example}
\newtheorem{question}[theorem]{Question}
\newtheorem{conjecture}[theorem]{Conjecture}
\newtheorem{algorithm}[theorem]{Algorithm}

\usepackage[bookmarksnumbered=true]{hyperref} 
\hypersetup{
     colorlinks = true,
     linkcolor = blue,
     anchorcolor = blue,
     citecolor = teal,
     filecolor = blue,
     urlcolor = blue
     }
\newcommand\restr[2]{{% we make the whole thing an ordinary symbol
  \left.\kern-\nulldelimiterspace % automatically resize the bar with \right
  #1 % the function
  \vphantom{\big|} % pretend it's a little taller at normal size
  \right|_{#2} % this is the delimiter
  }}

\everymath{\displaystyle}
\allowdisplaybreaks

\makeatletter
\def\mathcenterto#1#2{\mathclap{\phantom{#1}\mathclap{#2}}\phantom{#1}}
\let\old@widetilde\widetilde
\def\widetildeto#1#2{\mathcenterto{#2}{\old@widetilde{\mathcenterto{#1}{#2\,}}}}
\let\old@widehat\widehat
\def\widehatto#1#2{\mathcenterto{#2}{\old@widehat{\mathcenterto{#1}{#2\,}}}}
\makeatother

\newcommand{\size}[1]{\left| #1 \right|} % \size a b le pone | |
\newcommand{\set}[1]{{\left\{ #1 \right\}}} % \set a b le pone { }

\newcommand*\closure[1]{\overline{#1}}

\DeclareMathOperator{\rank}{rk}

\DeclareMathOperator{\CC}{\mathbb{C}}

\hypersetup{
    colorlinks=true,
    linkcolor=blue,
    filecolor=magenta,      
    urlcolor=cyan,
    pdftitle={Overleaf Example},
    pdfpagemode=FullScreen,
    }
    
\setlist[itemize]{noitemsep}

\normalfont

\usepackage{bbm}
\usepackage{dsfont}

\usepackage{url}
\usepackage{amsmath,blkarray,booktabs, bigstrut}
\usepackage{tikz}
\usepackage{multirow}
\usepackage{xcolor}
\usepackage{color}
\title{Minimal matroids in dependency posets: algorithms and applications to computing irreducible decompositions of circuit varieties

}
\author{Emiliano Liwski and Fatemeh Mohammadi}
\begin{document}
\maketitle

\begin{abstract}
We study point-line configurations, their minimal matroids, and their associated circuit varieties. We present an algorithm for identifying the minimal matroids of these configurations with respect to dependency order, or equivalently, the maximal matroids with respect to weak order, and use it to determine the irreducible decomposition of their corresponding circuit varieties. Our algorithm is applied to several classical configurations, including the Fano matroid, affine plane of order three, MacLane, and Pappus configurations. Additionally, we explore the connection to a conjecture by Jackson and Tanigawa, which provides a criterion for the uniqueness of the minimal matroids.  

\end{abstract}

{\hypersetup{linkcolor=black}
{\tableofcontents}}

%\vspace{50pt}

\section{Introduction}

A point-line configuration is a finite incidence structure consisting of elements called points and subsets called lines, which satisfy specific conditions. Equivalently, it can be viewed as a simple matroid of rank three. This concept arises in various fields, including the study of the connectivity of moduli spaces of %hyperplane 
line arrangements and the cohomology of their complement manifolds~\cite{rybnikov2011fundamental, nazir2012connectivity,Fatemeh6}, as well as the study of singularities and smoothness of %associated 
their varieties~\cite{vakil2006murphy, lee2013mnev}. 
They also arise in the study of determinantal varieties \cite{bruns2003determinantal}, as well as in the analysis of conditional independence models in statistics \cite{caines2022lattice, clarke2021matroid}. 
Of particular interest are $(v_r, b_k)$ configurations, which notably includes the family of Steiner systems \cite{gropp1991configurations, colbourn2006steiner, lindner2011topics}. For a detailed survey and key examples of these configurations, see \cite{gropp1997configurations}. % as examples.

Matroids, introduced by Whitney (see \cite{whitney1992abstract, Oxley, piff1970vector}), are combinatorial structures that generalize and unify various concepts of independence found in matrices and graphs. Given a matroid $M$, its rank is defined as the size of any maximal independent set of $M$. 
Point-line configurations are examples of matroids of rank three.
A finite collection of vectors in a fixed vector space determines a matroid through its linearly dependent sets. If this process is reversible, meaning that for a given matroid $M$, we can construct a collection of vectors that realizes $M$, these vectors are called a realization of $M$. The space of all realizations of $M$ is denoted by $\Gamma_M$, and its Zariski closure, referred to as the matroid variety, is denoted by $V_M$~\cite{gelfand1987combinatorial}.

In this work, we focus on point-line configurations, i.e.~matroids of rank at most three, which, when defined on a common ground set, form a poset under the dependency order defined by $M \leq N$, if every dependent set in $M$ is also dependent in $N$. This dependency order is the reverse of the {\em weak order} discussed in the literature~\cite{Oxley}. Within this framework, we present an algorithm for identifying the minimal matroids of a given point-line configuration $M$. These minimal matroids are the smallest matroids greater than $M$ in the dependency order.

On the geometric side, we focus on the circuit variety $V_{\mathcal{C}(M)}$ of a matroid $M$, which is associated with the circuits, i.e., the minimal dependent sets of $M$. The primary goal in this context is to determine the irreducible decomposition of $V_{\mathcal{C}(M)}$. In \cite{clarke2021matroid}, a decomposition strategy was proposed based on identifying the set of minimally dependent matroids for $M$. However, no explicit method for computing these minimal matroids was provided. In this work, we present an algorithm for identifying the minimal matroids and apply it to determine the irreducible decomposition of their circuit varieties.

For a general matroid $M$ with a large ground set, the algorithm is inefficient due to the large number of cases that arise. However, its efficiency improves significantly for symmetric and well-structured matroids. A notable example is the family of $(v_r, b_k)$ configurations; see Definition~\ref{configuration}. %, which we define below.

We apply our algorithm to several symmetric $v_3$ configurations from \cite{gropp1991configurations}, including the Fano plane, the MacLane configuration, the affine plane of order three, the Pappus configuration, and the second $9_3$ configuration. First, we compute their minimal matroids and then identify those that give rise to irreducible components in the decomposition of the associated circuit variety. Notably, these computations cannot be performed using existing computer algebra systems.

\smallskip
\noindent
{\bf Our contributions.} The main results of this paper are summarized as follows. A central notion in our work is the set of minimal matroids, which is formally introduced in Definition~\ref{minimal matroids}.

\begin{itemize}
\item We introduce an algorithm for identifying the set of all minimal matroids associated with a given point-line configuration, as detailed in Algorithm~\ref{algo m}. 
\item We propose a decomposition strategy for computing the irreducible decomposition of circuit varieties associated with point-line configurations, discussed in Section~\ref{strat}.
\item We compute the irreducible decomposition of circuit varieties for several classical configurations $v_{3}$, which are the Fano plane, the MacLane configuration, the affine plane of order three, the Pappus configuration, and the second $9_{3}$ configuration, as presented in Section~\ref{examples}.
\item Finally, we connect our work to a conjecture by Jackson and Tanigawa in \cite{jackson2024maximal} on the existence of a unique minimal $\mathcal{X}$-matroid and propose a refinement.
\end{itemize}

\noindent
{\bf Outline.} Section~\ref{sec 2} presents key notions, such as point-line configurations and their realization spaces. In Section~\ref{algorithm}, we present an algorithm to identify the minimal matroids of a given point-line configuration. Section~\ref{strat} proposes a decomposition strategy for determining the irreducible decomposition of circuit varieties of point-line configurations. In Section~\ref{examples}, we apply this strategy to compute the irreducible decomposition of circuit varieties for several classical symmetric $v_{3}$ configurations. In Section~\ref{quest}, we mention a related conjecture posed by Jackson and Tanigawa \cite{jackson2024maximal} and provide a counterexample and an alternative formulation of this conjecture. Finally, in Section~\ref{appen}, we introduce techniques for verifying redundant matroid varieties in the decomposition of matroid varieties and provide proofs of the technical lemmas from Section~\ref{examples}.

\medskip
\noindent{\bf Acknowledgement.}~We are grateful to Rémi Prébet for insightful discussions and for implementing the algorithm for computing minimal matroids.  F.M. would like to thank Guardo Elena and Maddie Weinstein for their discussion on the MacLane configuration. E.L. is supported by the FWO PhD fellowship 1126125N. F.M. was partially supported by the grants G0F5921N (Odysseus) and G023721N from the Research Foundation - Flanders (FWO) %, the UiT Aurora project MASCOT 
and the grant iBOF/23/064 from the KU Leuven.

\section{Preliminaries}\label{sec 2}

In this section, we briefly recall properties of matroids and their associated varieties, referring to \cite{Oxley, gelfand1987combinatorial, Fatemeh4, Fatemeh5} for details. We use $[n] = \{1, \ldots, n\}$ and $\textstyle \binom{[d]}{n}$ for the set of $n$-subsets of $[d]$.

\subsection{Matroids}

There are several equivalent definitions of matroids, each suited to different contexts. Here, we focus on the circuit definition. For a broader introduction to matroid theory, see \cite{Oxley, piff1970vector}.
\begin{definition}\normalfont
A matroid $M$ consists of a ground set $[d]$ together with a collection $\mathcal{C}$ of subsets of $[d]$, called
circuits, that satisfy the following three axioms: %Given a collection $\mathcal{C}$ of subsets of $[n]$:
\begin{itemize}
\item $\emptyset \not\in\mathcal{C}$.
\item If $C_1,C_2 \in \mathcal{C}$ and $C_1\subseteq C_2$, then $C_1=C_2$.
\item  %Circuit elimination:
 If $e\in C_1\cap C_2$ with $C_{1}\neq C_{2}\in \mathcal{C}$, there exists $C_3\in\mathcal{C}$ such that $C_3\subseteq (C_1\cup C_2)\backslash\{e\}$. 
\end{itemize}
We denote the circuits of $M$ by $\mathcal{C}(M)$, with $\mathcal{C}_{i}(M)$ representing the collection of circuits of size $i$.
\end{definition}

We will require the following terminology.

\begin{definition}\normalfont
Let $M$ be a matroid on the ground set $[d]$.
\begin{itemize}
\item A subset of $[d]$ that contains a circuit is called {\em dependent}, while any subset that does not contain a circuit is {\em independent}. The collection of all dependent sets of $M$ is denoted by $\mathcal{D}(M)$. 
\item A {\em basis} is a maximal independent subset of $[d]$, with respect to inclusion. The set of all bases is denoted by $\mathcal{B}(M)$.
%\item The set of circuits of $M$ is denoted by $\mathcal{C}(M)$.
\item For any $F\subset [d]$, the {\em rank} of $F$, denoted $\rank(F)$, is the size of the largest independent set contained in $F$. Note that $\rank([d])$ is the size of any basis. We define $\rank(M)$  as $\rank([d])$.
\item The {\em automorphism group} of $M$, denoted by $\text{Aut}(M)$, is the subgroup of all permutations $\sigma \in \mathbb{S}_d$ such that $X \in \mathcal{D}(M)$ if and only if $\sigma(X) \in \mathcal{D}(M)$.  
\end{itemize}
\end{definition}

Given a matroid $M$ on the ground set $[d]$, we define two operations on $M$: restriction and deletion.
\begin{definition}%[Submatroid]
\normalfont \label{subm}
For any subset $S \subseteq [d]$, the {\em restriction} of $M$ to $S$ is the matroid %$\restr{M}{S}$ on $S$ 
on $S$ whose rank function is the restriction of the rank function of $M$ to $S$. This matroid is referred to as the restriction of $M$ to $S$. 
Unless otherwise specified, we assume that subsets of $[d]$ possess this structure and refer to them as \textit{submatroids} of $M$. This submatroid is denoted by $M|S$, or simply $S$ when the context is clear. The {\em deletion} of $S$, is denoted by $M\setminus S$, which corresponds to $M|([d]\setminus S)$.
%\end{itemize}
\end{definition}

\begin{example}\normalfont\label{uniform 3}
A uniform matroid $U_{n, d}$ on the ground set $[d]$ of rank $n$ is defined as follows: every subset $S\subset [d]$ with $|S| \leq n$ is independent, while those with $|S| > n$ are dependent. 
\end{example}

\subsection{Point-line configurations}

\begin{definition}\label{point-line}
A point-line configuration $M$ consists of a ground set $[d]$ and a set of lines $\mathcal{L}\subset 2^{[d]}$ such that each line contains at least three distinct elements of $[d]$, and any pair of elements of $[d]$ lies on at most one common line. Equivalently, these can be described as simple matroids of rank at most three. For such a matroid $M$, we adopt the following terminology:

\begin{itemize}
\item A {\em line} is defined as a maximal dependent subset of $[d]$, where every subset of three elements is dependent.  We denote the set of all lines of $M$ by $\mathcal{L}_M$, or simply $\mathcal{L}$ when the context is clear.

\item Elements in $[d]$ are called {\em points}. For any point $p \in [d]$, let $\mathcal{L}_{p}$ denote the set of lines containing $p$. The {\em degree} of $p$ is defined as $\size{\mathcal{L}_{p}}$.

\end{itemize}
\end{definition}

\begin{example}\label{three lines}
Consider the point-line configuration depicted in Figure~\ref{fig:combined} (Left). The set of points is $[7]$, and the lines are $\mathcal{L}=\{   \{{1},{2},{7}\} , \{{3},{4},{7}\}  , \{{5},{6},{7}\} \}$. For the quadrilateral set $\text{QS}$, shown in Figure~\ref{fig:combined} (Right), the set of points is $[6]$ and $\mathcal{L}_{\text{QS}}=\{\{1,2,3\},\{1,5,6\},\{2,4,6\},\{3,4,5\}\}$.
\end{example}

\begin{figure}[h]
    \centering
    \begin{subfigure}[b]{0.3\textwidth}
        \centering
        \begin{tikzpicture}[x=0.75pt,y=0.75pt,yscale=-1,xscale=1]

\tikzset{every picture/.style={line width=0.75pt}} %set default line width to 0.75pt        

%Straight Lines [id:da06101063982657062] 
\draw    (81.69,116.61) -- (191.16,174.3) ;
%Straight Lines [id:da7678393526214486] 
\draw    (77,131.88) -- (224,131.88) ;
%Straight Lines [id:da3419905968139849] 
\draw    (80.13,150.55) -- (191.16,79.27) ;
%Shape: Ellipse [id:dp43020942362366354] 
\draw [fill={rgb, 255:red, 173; green, 216; blue, 230}, fill opacity=1]
(107.34,131.2) .. controls (107.34,133.07) and (108.63,134.58) .. (110.23,134.58) .. controls (111.83,134.58) and (113.12,133.07) .. (113.12,131.2) .. controls (113.12,129.33) and (111.83,127.81) .. (110.23,127.81) .. controls (108.63,127.81) and (107.34,129.33) .. (107.34,131.2) -- cycle ;
%Shape: Ellipse [id:dp7909486902309077] 
\draw [fill={rgb, 255:red, 173; green, 216; blue, 230}, fill opacity=1]
(142.62,149.93) .. controls (142.62,151.8) and (143.91,153.32) .. (145.51,153.32) .. controls (147.11,153.32) and (148.4,151.8) .. (148.4,149.93) .. controls (148.4,148.06) and (147.11,146.55) .. (145.51,146.55) .. controls (143.91,146.55) and (142.62,148.06) .. (142.62,149.93) -- cycle ;
%Shape: Ellipse [id:dp6735689153424713] 
\draw [fill={rgb, 255:red, 173; green, 216; blue, 230}, fill opacity=1]
(173.7,166.79) .. controls (173.7,168.66) and (174.99,170.18) .. (176.59,170.18) .. controls (178.19,170.18) and (179.48,168.66) .. (179.48,166.79) .. controls (179.48,164.92) and (178.19,163.41) .. (176.59,163.41) .. controls (174.99,163.41) and (173.7,164.92) .. (173.7,166.79) -- cycle ;
%Shape: Ellipse [id:dp9808673721534766] 
\draw [fill={rgb, 255:red, 173; green, 216; blue, 230}, fill opacity=1] (201.42,131.2) .. controls (201.42,133.07) and (202.71,134.58) .. (204.31,134.58) .. controls (205.91,134.58) and (207.2,133.07) .. (207.2,131.2) .. controls (207.2,129.33) and (205.91,127.81) .. (204.31,127.81) .. controls (202.71,127.81) and (201.42,129.33) .. (201.42,131.2) -- cycle ;
%Shape: Ellipse [id:dp8560955724937005] 
\draw [fill={rgb, 255:red, 173; green, 216; blue, 230}, fill opacity=1] (167.82,131.2) .. controls (167.82,133.07) and (169.11,134.58) .. (170.71,134.58) .. controls (172.31,134.58) and (173.6,133.07) .. (173.6,131.2) .. controls (173.6,129.33) and (172.31,127.81) .. (170.71,127.81) .. controls (169.11,127.81) and (167.82,129.33) .. (167.82,131.2) -- cycle ;
%Shape: Ellipse [id:dp8522323991300782] 
\draw [fill={rgb, 255:red, 173; green, 216; blue, 230}, fill opacity=1] (177.06,87.18) .. controls (177.06,89.05) and (178.35,90.56) .. (179.95,90.56) .. controls (181.55,90.56) and (182.84,89.05) .. (182.84,87.18) .. controls (182.84,85.31) and (181.55,83.79) .. (179.95,83.79) .. controls (178.35,83.79) and (177.06,85.31) .. (177.06,87.18) -- cycle ;
%Shape: Ellipse [id:dp38553240573402947] 
\draw [fill={rgb, 255:red, 173; green, 216; blue, 230}, fill opacity=1] (145.14,105.91) .. controls (145.14,107.78) and (146.43,109.3) .. (148.03,109.3) .. controls (149.63,109.3) and (150.92,107.78) .. (150.92,105.91) .. controls (150.92,104.04) and (149.63,102.52) .. (148.03,102.52) .. controls (146.43,102.52) and (145.14,104.04) .. (145.14,105.91) -- cycle ;

% Text Node
\draw (166.25,71.57) node [anchor=north west][inner sep=0.75pt]   [align=left] {{\scriptsize $1$}};
% Text Node
\draw (206.28,115.51) node [anchor=north west][inner sep=0.75pt]   [align=left] {{\scriptsize $3$}};
% Text Node
\draw (138.91,153.35) node [anchor=north west][inner sep=0.75pt]   [align=left] {{\scriptsize $6$}};
% Text Node
\draw (170.08,173.62) node [anchor=north west][inner sep=0.75pt]   [align=left] {{\scriptsize $5$}};
% Text Node
\draw (173.44,116.45) node [anchor=north west][inner sep=0.75pt]   [align=left] {{\scriptsize $4$}};
% Text Node
\draw (106.8,111.36) node [anchor=north west][inner sep=0.75pt]   [align=left] {{\scriptsize $7$}};
% Text Node
\draw (136.66,90.59) node [anchor=north west][inner sep=0.75pt]   [align=left] {{\scriptsize $2$}};
% Text Node
%\draw (96.73,190.03) node [anchor=north west][inner sep=0.75pt]   [align=left] {{\footnotesize Pencil of three lines}};

   \end{tikzpicture}
        %\caption{Quadrilateral set}
        \label{fig:quadrilateral 2}
    \end{subfigure}
    %\hfill
    \begin{subfigure}[b]{0.3\textwidth}
        \centering

\tikzset{every picture/.style={line width=0.75pt}} %set default line width to 0.75pt        

\begin{tikzpicture}[x=0.75pt,y=0.75pt,yscale=-1,xscale=1]

%Straight Lines [id:da9613256122302601] 
\draw [line width=0.75]    (247.65,123.92+60) -- (201.96,207.53+60) ;
%Straight Lines [id:da3259117734461473] 
\draw [line width=0.75]    (247.65,123.92+60) -- (291.68,207.53+60) ;
%Straight Lines [id:da2624677482334885] 
\draw [line width=0.75]    (219.23,174.98+60) -- (291.68,207.53+60) ;
%Straight Lines [id:da26466266747028633] 
\draw [line width=0.75]    (274.41,174.98+60) -- (201.96,207.53+60) ;
%Shape: Ellipse [id:dp10345993152561672] 
\draw [fill={rgb, 255:red, 173; green, 216; blue, 230}, fill opacity=1] (244.87,123.92+60) .. controls (244.87,125.68+60) and (246.12,127.11+60) .. (247.65,127.11+60) .. controls (249.19,127.11+60) and (250.44,125.68+60) .. (250.44,123.92+60) .. controls (250.44,122.15+60) and (249.19,120.73+60) .. (247.65,120.73+60) .. controls (246.12,120.73+60) and (244.87,122.15+60) .. (244.87,123.92+60) -- cycle ;
%Shape: Ellipse [id:dp21945184574290333] 
\draw [fill={rgb, 255:red, 173; green, 216; blue, 230}, fill opacity=1] (244.31,187.1+60) .. controls (244.31,188.87+60) and (245.56,190.29+60) .. (247.1,190.29+60) .. controls (248.64,190.29+60) and (249.88,188.87+60) .. (249.88,187.1+60) .. controls (249.88,185.34+60) and (248.64,183.91+60) .. (247.1,183.91+60) .. controls (245.56,183.91+60) and (244.31,185.34+60) .. (244.31,187.1+60) -- cycle ;
%Shape: Ellipse [id:dp033802625652762264] 
\draw [fill={rgb, 255:red, 173; green, 216; blue, 230}, fill opacity=1] (216.45,174.98+60) .. controls (216.45,176.74+60) and (217.69,178.17+60) .. (219.23,178.17+60) .. controls (220.77,178.17+60) and (222.02,176.74+60) .. (222.02,174.98+60) .. controls (222.02,173.21+60) and (220.77,171.79+60) .. (219.23,171.79+60) .. controls (217.69,171.79+60) and (216.45,173.21+60) .. (216.45,174.98+60) -- cycle ;
%Shape: Ellipse [id:dp395767307047196] 
\draw [fill={rgb, 255:red, 173; green, 216; blue, 230}, fill opacity=1] (271.62,174.98+60) .. controls (271.62,176.74+60) and (272.87,178.17+60) .. (274.41,178.17+60) .. controls (275.94,178.17+60) and (277.19,176.74+60) .. (277.19,174.98+60) .. controls (277.19,173.21+60) and (275.94,171.79+60) .. (274.41,171.79+60) .. controls (272.87,171.79+60) and (271.62,173.21+60) .. (271.62,174.98+60) -- cycle ;
%Shape: Ellipse [id:dp5069700831588411] 
\draw [fill={rgb, 255:red, 173; green, 216; blue, 230}, fill opacity=1] (288.9,207.53+60) .. controls (288.9,209.29+60) and (290.14,210.72+60) .. (291.68,210.72+60) .. controls (293.22,210.72+60) and (294.47,209.29+60) .. (294.47,207.53+60) .. controls (294.47,205.77+60) and (293.22,204.34+60) .. (291.68,204.34+60) .. controls (290.14,204.34+60) and (288.9,205.77+60) .. (288.9,207.53+60) -- cycle ;
%Shape: Ellipse [id:dp3796249357557302] 
\draw [fill={rgb, 255:red, 173; green, 216; blue, 230}, fill opacity=1] (199.17,207.53+60) .. controls (199.17,209.29+60) and (200.42,210.72+60) .. (201.96,210.72+60) .. controls (203.49,210.72+60) and (204.74,209.29+60) .. (204.74,207.53+60) .. controls (204.74,205.77+60) and (203.49,204.34+60) .. (201.96,204.34+60) .. controls (200.42,204.34+60) and (199.17,205.77+60) .. (199.17,207.53+60) -- cycle ;

% Text Node
\draw (243.29,107.07+60) node [anchor=north west][inner sep=0.75pt]   [align=left] {{\scriptsize $1$}};
% Text Node
\draw (200.77,165.67+60) node [anchor=north west][inner sep=0.75pt]   [align=left] {{\scriptsize $2$}};
% Text Node
\draw (184.94,205.77+60) node [anchor=north west][inner sep=0.75pt]   [align=left] {{\scriptsize $3$}};
% Text Node
\draw (241.88,169.35+60) node [anchor=north west][inner sep=0.75pt]   [align=left] {{\scriptsize $4$}};
% Text Node
\draw (277.98,164.77+60) node [anchor=north west][inner sep=0.75pt]   [align=left] {{\scriptsize $5$}};
% Text Node
\draw (292.84,206.21+60) node [anchor=north west][inner sep=0.75pt]   [align=left] {{\scriptsize $6$}};
% Text Node
%\draw (203.1,213.83) node [anchor=north west][inner sep=0.75pt]   [align=left] {{\footnotesize Quadrilateral set}};

\end{tikzpicture}
       % \caption{Pascal configuration}
        \label{fig:pascal}
    \end{subfigure}
\caption{(Left) Three concurrent lines; (Right) Quadrilateral set.}
    \label{fig:combined}
\end{figure}
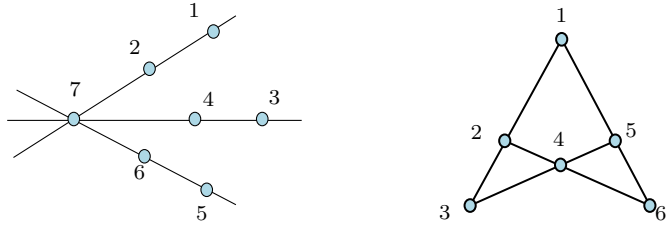

A notable subclass of point-line configurations is formed by the configurations $(v_{r},b_{k})$; see~e.g.~\cite{gropp1997configurations}. 

\begin{definition}\label{configuration}
A $(v_r, b_k)$ {\em configuration} %$(v_{r},b_{k})$ 
is a point-line configuration such that: %satisfying the following conditions:
\begin{itemize}
\item There are $v$ points and $b$ lines.
\item Each line contains $k$ points, and each point lies on $r$ lines.
\end{itemize}
When $v=b$ (and hence $r=k$), the configuration is called symmetric and is denoted by $v_{k}$.   
\end{definition}

\begin{example}
The quadrilateral set $\text{QS}$ in Figure~\ref{fig:combined} (Right) is a $(6_{2},4_{3})$ configuration, since every line contains $3$ points and each point lies on $2$ lines.
\end{example}

A particular subfamily of the configurations $(v_{r},b_{k})$ consists of those derived from a {\em Steiner system}. 

\begin{example}
A {\em Steiner system} with parameters $t,k,n$, denoted $S(t,k,n)$, consists of an $n$-element set $S$ together with a collection of $k$-element subsets of $S$ known as {\em blocks}, such that each $t$-element subset is contained in exactly one block.

Each Steiner system $S(2,k,n)$ represents a configuration $(v_{r},b_{k})$ where $S$ is the ground set and the blocks are the lines. Thus, the configurations $(v_r, b_k)$ generalize Steiner systems with $t=2$.
\end{example}

We study two families of point-line configurations from \cite{Fatemeh5} that play a key role in this note.

\begin{definition}\normalfont 
For a point-line configuration $M$ on $[d]$, we define 
\[S_{M}=\{p\in [d]: \size{\mathcal{L}_{p}}\geq 2\},\quad \text{and} \quad Q_{M}=\{p\in [d]: \size{\mathcal{L}_{p}}\geq 3\}.\] 
We then consider the following chains of submatroids of $M$: 
\begin{equation*}
\begin{aligned}
&M_{0}=M, \ \quad M_{1}=S_{M},\quad \text{and }\quad M_{j+1}=S_{M_{j}} \quad \ \text{ for all $j\geq 1$.}\\
& M^{0}=M,\quad M^{1}=Q_{M},\quad \text{and }\quad M^{j+1}=Q_{M^{j}} \quad \text{ for all $j\geq 1$.}
\end{aligned}
\end{equation*}
We say that $M$ is {\em nilpotent} if $M_{j}=\emptyset$ for some $j$, and 
{\em solvable} if $M^{j}=\emptyset$ for some $j$.
\end{definition}

\subsection{Realization spaces of matroids, circuit and matroid varieties}

We begin by recalling the realization spaces of matroids, followed by circuit and matroid varieties.  
\begin{definition}\normalfont
Let $M$ be a matroid of rank $n$ on the ground set $[d]$. A realization of $M$ is a collection of vectors $\gamma=\{\gamma_{1},\ldots,\gamma_{d}\}\subset \CC^{n}$ such that
\[\{\gamma_{i_{1}},\ldots,\gamma_{i_{p}}\}\ \text{is linearly dependent} \Longleftrightarrow \{i_{1},\ldots,i_{p}\} \ \text{is a dependent set of $M$.}\]
The {\em realization space} of $M$ is defined as 
$\Gamma_{M}=\{\gamma\subset \CC^{n}: \gamma \ \text{is a realization of $M$}\}$.
Each element of $\Gamma_{M}$ corresponds to an $n\times d$ matrix over $\CC$. A matroid is called realizable if its realization space is non-empty. The {\em matroid variety} $V_M$ of $M$ is defined as the Zariski closure in $\CC^{nd}$ of  $\Gamma_M$.
\end{definition}

\begin{definition}\normalfont\label{cir}
Let $M$ be a matroid of rank $n$ on the ground set $[d]$.
We say that $\gamma$, a collection of vectors of $\CC^{n}$ indexed by $[d]$, includes the dependencies of $M$ if it satisfies:
\[\set{i_{1},\ldots,i_{k}}\  \text{is a dependent set of $M$} \Longrightarrow \set{\gamma_{i_{1}},\ldots,\gamma_{i_{k}}}\ \text{is linearly dependent}. \] 
The {\em circuit variety} of $M$ is defined as \[V_{\mathcal{C}(M)}=\{\gamma:\text{$\gamma$ includes the dependencies of $M$}\}.\]
\end{definition}

We recall the following result from \cite{Fatemeh5,Fatemeh4} on nilpotent and solvable point-line configurations, which will be used in the subsequent sections.  
\begin{theorem}\label{nil coincide}
Let $M$ be a point-line configuration on $[d]$. %with no points of degree greater than two.
\begin{itemize}
\item[{\rm (i)}] If $M$ is nilpotent and has no points of degree greater than two, then $V_{\mathcal{C}(M)}=V_{M}$. 
\item[{\rm (ii)}] If $M$ has no points of degree greater than two, and every proper submatroid is nilpotent, then $V_{\mathcal{C}(M)}=V_{M}\cup V_{U_{2,d}}$.
Here, $U_{n, d}$ is the uniform matroid of rank $n$ on the ground set $[d]$.

\item[{\rm (iii)}] If $M$ is solvable, then $V_{M}$ is either irreducible or empty.
\end{itemize}
\end{theorem}

\begin{example}\label{ej quad}
Consider the point-line configuration $\text{QS}$ illustrated in Figure~\ref{fig:combined} (Right). Every proper submatroid of $M$ is nilpotent, which implies that $V_{\mathcal{C}(\text{QS})}=V_{\text{QS}}\cup V_{U_{2,6}}$. 
\end{example}

%%%%%%%%%%%%%%%%%%%%

\section{Algorithm for identifying the set of minimal matroids}\label{algorithm}

In this section, we introduce an algorithm for identifying the minimal matroids of a given point-line configuration. Throughout, $M$ denotes a fixed point-line configuration on ground set $[d]$ with lines $\mathcal{L}$, and $\mathcal{D}(M)$ denotes the set of dependencies of $M$. We begin by establishing our notations.

\begin{definition}\label{dependency}
Let $N_{1}$ and $N_{2}$ be matroids on $[d]$. We say that $N_{1}\leq N_{2}$ if $\mathcal{D}(N_{1})\subset \mathcal{D}(N_{2})$. This partial order is referred to as the {\em dependency order} on matroids.  
\end{definition}

\begin{notation}
For a poset $P=(X,\leq)$, we establish the following terminology:

\begin{itemize}
\item For $x\in X$, $P_{>x}$ denotes the set $\{y\in X: y>x\}$. 
\item For a subset $Y\subset X$, $\min\{Y\}$ refers to the set of minimal elements of $Y$.

\end{itemize}
\end{notation}

\begin{definition}
A hypergraph $\Delta$ on the vertex set $[d]$ is defined as a collection of subsets of $[d]$, with the condition that no proper subset of any element in $\Delta$ belongs to $\Delta$. The elements of $\Delta$ are referred to as edges. We say that $\Delta_{1}\leq \Delta_{2}$ if, for each edge $e_{1}\in \Delta_{1}$, there exists an edge $e_{2}\in \Delta_{2}$ such that $e_{1}\subset e_{2}$. 
Let $\mathcal{H}_{\geq 3}$ denote the set of all hypergraphs in which every edge has size at least $3$. Note that if $M$ is a point-line configuration, then its set of lines $\mathcal{L}_{M}$ belongs to $\mathcal{H}_{\geq 3}$. 
\end{definition}

If $\Delta$ is a hypergraph on the vertex set $[d]$, we say that $\Delta\leq N$ if $\Delta\subseteq \mathcal{D}(N)$.  

\begin{definition}\label{minimal matroids}
 Let $\mathcal{M}$ be the poset of matroids on the ground set $[d]$, equipped with the dependency order. We define the set of all minimal matroids of $M$ as: 
\[\min(M)=\min\{\mathcal{M}_{>M}\}.\]
We also fix the following subsets of $\mathcal{M}_{>M}$:
\begin{equation*}
\begin{aligned}
&\mathcal{A}=\{N>M: \text{$\mathcal{C}_{1}(N)=\emptyset$ and $\mathcal{C}_{2}(N)\neq \emptyset$}\},\\
&\mathcal{B}=\{N>M: \text{$\mathcal{C}_{1}(N)=\emptyset$ and $\mathcal{C}_{2}(N)= \emptyset$}\},\\
&\mathcal{C}=\{N>M: \text{$\mathcal{C}_{1}(N)\neq \emptyset$}\},
\end{aligned}
\end{equation*}
where $\mathcal{C}_i(N)$ denotes the set of circuits,~i.e.,~minimal dependencies of size $i$ in $N$. Note that 
%\begin{equation*}
$\mathcal{M}_{>M}$ is the disjoint union of $\mathcal{A}$, $\mathcal{B} $ and $\mathcal{C}$, $\mathcal{M}_{>M}=\mathcal{A}\amalg \mathcal{B} \amalg \mathcal{C}$, 
%\end{equation*}
which leads to the following equality: 
\begin{equation}\label{unio}
\min(M)=\min\{\min\{\mathcal{A}\}\cup \min\{\mathcal{B}\}\cup \min\{\mathcal{C}\}\}.
\end{equation}
\end{definition}

In the following subsections, we first present algorithms for computing $\min\{\mathcal{A}\}$, $\min\{\mathcal{B}\}$, and $\min\{\mathcal{C}\}$, and then use them to develop an algorithm for computing $\min(M)$.

\subsection{Determining $\min\{\mathcal{A}\}$}
In this subsection, we present an algorithm to compute $\min\{\mathcal{A}\}$ from Definition~\ref{minimal matroids} and Equation~\eqref{unio}.

\begin{definition}\label{def:formula}
A {\em formula} on the set $[d]$ is a conjunction of atoms of the form $(x\sim y)$ or $(x\not \sim y)$, where $x,y\in [d]$. Explicitly, it can be written as: 
\begin{equation}\label{expr}
F=(a_{1,1}\sim a_{1,2})\ \wedge \cdots \wedge (a_{i,1}\sim a_{i,2}) \ \wedge (b_{1,1}\not \sim b_{1,2})\ \wedge \cdots \wedge (b_{j,1}\not \sim b_{j,2}).
\end{equation}
Let $\Pi$ denote the poset of all formulas on $[d]$, where $F\geq G$ if the atoms of $G$ are contained within those of $F$. Furthermore, define the following subsets:
\begin{equation*}
\begin{aligned}
&\Pi_{\sim}=\{F\in \Pi:\ \text{$F$ is composed solely of $\sim$}\},\\
&\Pi_{\not \sim}=\{F\in \Pi:\ \text{$F$ is composed solely of $\not \sim$}\}.
\end{aligned}
\end{equation*}
Note that each formula $F\in \Pi$ can be decomposed as 
$F=F_{\sim}\wedge F_{\not \sim}$
where $F_{\sim}\in \Pi_{\sim}$ and $F_{\not \sim}\in \Pi_{\not \sim}$. Note that $F_{\sim}$ defines an equivalence relation on $[d]$, and we denote the {\em set of its equivalence classes} by $[d]/F$. For any element $i\in [d]$, we denote its equivalence class by $\overline{i}_{F}$.
Additionally, we assume that these formulas are {\em consistent}, meaning that if $\overline{i}_{F}=\overline{j}_{F}$, then $(i\not \sim j)$ is not an atom of $F$. For a given $F\in \Pi$, as in \eqref{expr}, we define $\text{Reali}_{M}(F)$ as the set of all matroids  $N>M$ such that:
\begin{center}
    $\mathcal{C}_{1}(N)=\emptyset$,\quad $\{a_{k,1},a_{k,2}\}\in \mathcal{C}_{2}(N)$ for all $k\in [i]$,\quad and $\{b_{l,1},b_{l,2}\}\not \in \mathcal{C}_{2}(N)$ for all $l\in [j]$.
\end{center}  
\end{definition}

\begin{example}\label{exi 3}
Consider the point-line configuration $M_{\text{Fano}}$ depicted in Figure~\ref{new figure} (Left) and the formula \[F=(1\sim 2)\wedge (4\not \sim 7).\]
Let $N$ be the matroid of rank three on the ground set $[7]$ with circuits of size at most four, as follows:
\begin{equation*}
\mathcal{C}_{1}(N)=\emptyset, \quad \mathcal{C}_{2}(N)=\{\{i,j\}:\text{$i,j\in \{1,2,6,7\},i\neq j$}\}, \quad \text{and} \quad \mathcal{C}_{3}(N)=\{\{3,4,5\}\}.
\end{equation*}
Since $\{1,2\}\in \mathcal{C}_{2}(N),\{4,7\}\not \in \mathcal{C}_{2}(N)$ and $N>M_{\text{Fano}}$, we conclude that $N\in \text{Reali}_{M_{\text{Fano}}}(F)$.
\end{example}

We begin with a series of lemmas and definitions essential for outlining 
the algorithm to identify $\min\{\mathcal{A}\}$.

\begin{lemma}\label{unique 0}
Let $\Delta \in \mathcal{H}_{\geq 3}$ and $F\in \Pi_{\not \sim}$. There exists a unique hypergraph $\Delta_{F}$ that is minimal among those satisfying the following conditions:
\begin{itemize}
\item $\Delta\leq \Delta_{F}$.
\item There is no atom $(c_{1}\not \sim c_{2})$ in $F$ and no %two 
distinct edges $e_{1},e_{2}\in \Delta_{F}$ such that $\{c_{1},c_{2}\}\subset e_{1}\cap e_{2}$.
\end{itemize}
\end{lemma}

\begin{proof}
Let $\widetilde{\Delta}$ be a hypergraph that satisfies both conditions.
Suppose $e_{1},e_{2}\in \Delta$ are edges such that $e_{1}\cap e_{2}$ contains an atom of $F$. By hypothesis, some edge of $\widetilde{\Delta}$ must include $e_{1}\cup e_{2}$.
%This observation allows us to define an equivalence relation on the elements of $\Delta$, generated by $e_{1}\sim e_{2}$ if $e_{1}\cap e_{2}$ defines an atom of $\mathcal{P}$. 

Let $\Delta_{2}$ be the hypergraph obtained by adding $e_{1}\cup e_{2}$ to $\Delta$ and removing all the edges of $\Delta$ that are contained within $e_{1}\cup e_{2}$. From the earlier observation, it follows that $\widetilde{\Delta}$ satisfies $\Delta\leq \widetilde{\Delta}$ if and only if $ \Delta_{2}\leq \widetilde{\Delta}$. Repeating this process, we construct a chain of hypergraphs: 
\[\Delta\leq \Delta_{2}\leq \cdots,\]
where $\widetilde{\Delta}$ satisfies $\Delta \leq \widetilde{\Delta}$ if and only if $\Delta_{j}\leq \widetilde{\Delta}$ for any $j\geq 1$. 

Since $[d]$ is finite, this sequence of hypergraphs must eventually stabilize. Therefore, there exists $k>0$ such that $\Delta_{k}=\Delta_{k+1}$. At this point, it follows that $e_{1}\cap e_{2}$ does not include an atom of $F$ for any $e_{1},e_{2}\in \Delta_{k}$. Thus, $\Delta_{k}$ is the desired hypergraph.
\end{proof}

\begin{example}
Consider the hypergraph $\Delta=\{\{1,2,3\},\{2,3,4\},\{1,4,5\},\{1,5,6\}\}$ on $[6]$ and the formula $F=(2\not \sim 3)\wedge (1\not \sim 4)$. Under this formula, we have $\Delta_{F}=\{\{1,2,3,4,5\},\{1,5,6\}\}$.
\end{example}

Following the reasoning in Lemma~\ref{unique 0}, we propose the following algorithm to determine $\Delta_{F}$.

\begin{algorithm}\label{delta p}
Let $\Delta \in \mathcal{H}_{\geq 3}$ and $F\in \Pi_{\not \sim}$. To determine $\Delta_{F}$, proceed as follows:

\medskip
{\bf Initialization:} Begin with the hypergraph $\Delta_{1}=\Delta$.

\medskip
{\bf Recursion:} At the step $j$, consider the hypergraph $\Delta_{j}$ and perform the following:

\begin{itemize}
\item If there is no atom $(c_{1}\not \sim c_{2})$ in $F$ and no distinct edges $e_{1},e_{2}\in \Delta_{j}$ such that $\{c_{1},c_{2}\}\subset e_{1}\cap e_{2}$, set $\Delta_{j}=\Delta_{F}$ and terminate the process.
\item If there exists an atom $(c_{1}\not \sim c_{2})$ in $F$ and distinct edges $e_{1},e_{2}\in \Delta_{j}$ such that $\{c_{1},c_{2}\}\subset e_{1}\cap e_{2}$, define $\Delta_{j+1}$ as the hypergraph obtained by adding $e_{1}\cup e_{2}$ to $\Delta_{j}$, removing all the edges contained within $e_{1}\cup e_{2}$, and continue the process.
\end{itemize}

{\bf Termination:} By Lemma~\ref{unique 0}, this process concludes with $\Delta_{F}$.

\end{algorithm}

By applying Lemma~\ref{unique 0}, we obtain the following corollary.

\begin{lemma}\label{unique}
Let $\Delta\in \mathcal{H}_{\geq 3}$. There exists a unique point-line configuration $N$ on the ground set $[d]$ that is minimal with respect to the condition $N\geq \Delta$. 
\end{lemma}

\begin{proof}
Applying Lemma~\ref{unique} to the formula
$\textstyle{F=\bigwedge_{i\neq j\in [d]} (i\not \sim j)}$,
we conclude that there exists a unique hypergraph $\Delta_{F}$ that is minimal with respect to the following conditions:
\begin{itemize}
\item $\Delta \leq \Delta_{F}$.
\item Any two distinct edges $e_{1},e_{2}\in \Delta_{F}$ satisfy $\size{e_{1}\cap e_{2}}\leq 1$.
\end{itemize}
Thus, by the second condition, $\Delta_{F}$ defines the lines of a point-line configuration $N$. Moreover, due to the minimality condition on $\Delta_{F}$, $N$ is the minimal point-line configuration satisfying $N\geq \Delta$.
\end{proof}

\begin{example}
Consider the hypergraph $\Delta=\{\{1,2,3\},\{1,5,6\},\{1,2,5\},\{6,7,8\},\{6,7,9,10\}\}$ on $[10]$. The minimal point-line configuration $N$ on $[10]$ that satisfies $N\geq \Delta$ has lines
\[\mathcal{L}_{N}=\{\{1,2,3,5,6\},\{6,7,8,9,10\}\}.\]
\end{example}

\begin{definition}\label{new def}
Let $\Delta \in \mathcal{H}_{\geq 3}$ and $F\in \Pi_{\sim}$. We define the hypergraph $\Delta_{F}\in \mathcal{H}_{\geq 3}$ on the vertex set $[d]/F$, the set of equivalence classes as in Definition~\ref{def:formula}, as follows:
\[\Delta_{F}=\{\{\overline{i}_{F}: i\in e\}: \text{$e\in \Delta$ and $\size{\{\overline{i}_{F}: i\in e\}}\geq 3$}\}.\]
\end{definition}

We now generalize Lemma~\ref{unique 0} and Definition~\ref{new def} that enables us to define $\Delta_{F}$ for any $F\in \Pi$.

\begin{lemma}\label{unique 35}
Let $\Delta\in \mathcal{H}_{\geq 3}$ and $F\in \Pi$. There exists a unique minimal hypergraph $\Delta_{F}$ on the ground set $[d]/F$, the set of equivalence classes as in Definition~\ref{def:formula}, satisfying the following conditions:
\begin{itemize}
\item For any $e\in \Delta$ with $\size{\{\overline{i}_{F}: i\in e\}}\geq 3$, there exists $e\rq \in \Delta_{F}$ satisfying  $\{\overline{i}_{F}: i\in e\}\subset e\rq$. 
\item No atom $(c_{1}\not \sim c_{2})$ in $F_{\not \sim}$ and no pair $e_{1},e_{2}\in \Delta_{F}$ of distinct edges 
 satisfy $\{\overline{c_{1}}_{F},\overline{c_{2}}_{F}\}\subset e_{1}\cap e_{2}$.
\end{itemize}
Furthermore, we have 
$\Delta_{F}=(\Delta_{F_{\sim}})_{F_{\not\sim}}$.
\end{lemma}

\begin{proof}
Let $\widetilde{\Delta}$ be a hypergraph satisfying both conditions.
From Definition~\ref{new def}, we note that the first  condition is equivalent to 
$\textstyle{\Delta_{F_{\sim}} \leq \widetilde{\Delta}}$.
By Lemma~\ref{unique 0}, there exists a unique hypergraph $\Delta_{F}$ on the ground set $[d]/F$ that is minimal with respect to the following conditions:
\begin{itemize}
\item $\Delta_{F_{\sim}} \leq \Delta_{F}$.
\item There is no atom $(c_{1}\not \sim c_{2})$ in $F_{\not \sim}$ and no distinct edges $e_{1},e_{2}\in \Delta_{F}$ such that $\{\overline{c_{1}}_{F},\overline{c_{2}}_{F}\}\subset e_{1}\cap e_{2}$.
\end{itemize}
Moreover, by definition, we have
$\Delta_{F}=(\Delta_{F_{\sim}})_{F_{\not\sim}}$.
\end{proof}

\begin{remark}\label{recursive}
By applying Lemma~\ref{unique 35}, we can recursively compute $\Delta_{F}$ for $\Delta\in \mathcal{H}_{\geq 3}$ and $F\in \Pi$. Specifically, we have:
\begin{itemize}
\item If $F=F\rq \wedge (x\not \sim y)$, then
$\Delta_{F}=(\Delta_{F\rq})_{F_{\not \sim}}$.
\item If $F=F\rq \wedge (x\sim y)$, then
$\Delta_{F}=((\Delta_{F\rq})_{(x\sim y)})_{F_{\not \sim}}$.
\end{itemize}
\end{remark}

\begin{example}\label{exi}
Let $\Delta=\{\{1,2,3\},\{1,5,6\},\{1,4,7\},\{2,5,7\},\{2,4,6\},\{3,4,5\},\{3,6,7\}\}$ be the hypergraph, whose edges are the lines of the %point-line 
configuration $M_{\text{Fano}}$ in Figure~\ref{new figure} (Left). Let $F$ be the formula $(1\sim 2)\wedge (1\not \sim 5)\wedge (1\not \sim 4)$. Following the steps of Remark~\ref{recursive}, we compute $\Delta_{F}$ as follows:
\begin{itemize}
\item We identify $1$ and $2$, yielding the hypergraph
\[\Delta_{(1\sim 2)}=\{\{1,5,6\},\{1,4,7\},\{1,5,7\},\{1,4,6\},\{3,4,5\},\{3,6,7\}\}.\]
\item Next, we merge $\{1,5,6\}$ and $\{1,5,7\}$, obtaining the hypergraph
\[\Delta_{(1\sim 2)\wedge (1\not \sim 5)}=\{\{1,5,6,7\},\{1,4,7\},\{1,4,6\},\{3,4,5\},\{3,6,7\}\}.\]
\item Finally, we merge $\{1,4,7\}$ and $\{1,4,6\}$, resulting in the hypergraph
\[\Delta_{F}=\{\{1,5,6,7\},\{1,4,6,7\},\{3,4,5\},\{3,6,7\}\}.\]
\end{itemize}
\end{example}

We now present a fundamental property that applies to certain elements $F\in \Pi$. 

\begin{definition}\label{property X}
For a point-line configuration $M$ with the set of lines $\mathcal{L}$, considered as a hypergraph, we say that a formula $F\in \Pi$ has the property $X$ if:
\begin{itemize}
\item For any distinct pair $e_{1},e_{2}\in (\mathcal{L})_{F}$, it holds that $\size{e_{1}\cap e_{2}}\leq 1$. 
\end{itemize}
For a fixed $M$, the set of all formulas $F\in \Pi$ satisfying this condition is denoted by $\Pi_{X}$.
\end{definition}

\begin{example}\label{exi 2}
Consider the hypergraph $\Delta$ in Example~\ref{exi} and the matroid $M_{\text{Fano}}$. For the formula \[F=(1\sim 2)\wedge (1\sim 6)\wedge (1\sim 7),\] we obtain $\Delta_{F}=\{\{3,4,5\}\}$, which implies that $F\in \Pi_{X}$.
\end{example}

With the notation above, we have the following lemma for the formulas in $\Pi_{X}$.

\begin{lemma}\label{unico}
For %a point-line configuration $M$ and 
each $F\in \Pi_{X}$, there exists a unique minimal matroid in $\textup{Reali}_{M}(F)$. 
\end{lemma}

\begin{proof}
By Lemma~\ref{unique 35}, $(\mathcal{L})_{F}$ is the unique minimal hypergraph satisfying
the following conditions:
\begin{itemize}
\item For any $l\in \mathcal{L}$ with $\size{\{\overline{i}_{F}: i\in l\}}\geq 3$, there exists an edge $e \in (\mathcal{L})_{F}$ satisfying  $\{\overline{i}_{F}: i\in l\}\subset e$. 
\item There is no atom $(c_{1}\not \sim c_{2})$ in $F_{\not \sim}$ and no distinct edges $e_{1},e_{2}\in (\mathcal{L})_{F}$ such that \[\{\overline{c_{1}}_{F},\overline{c_{2}}_{F}\}\subset e_{1}\cap e_{2}.\]
\end{itemize}

We define the matroid $N\in \mathcal{A}$ as follows:
\begin{equation*}
\mathcal{C}_{2}(N)=\{\{i,j\}:\text{$\overline{i}_{F}=\overline{j}_{F}$}\}, \quad  \mathcal{C}_{3}(N)=\{\{i,j,k\}: \text{there exists $e\in (\mathcal{L})_{F}$ such that $\{\overline{i}_{F},\overline{j}_{F},\overline{k}_{F}\}\subset e$}\}
\end{equation*}
Since $F\in \Pi_{X}$, the collection $(\mathcal{L})_{F}$ corresponds to the set of lines of a point-line configuration on the ground set $[d]/F$. It is thus straightforward to verify that $N$ is well defined and belongs to $\text{Reali}_{M}(F)$. Furthermore, by the minimality condition on $(\mathcal{L})_{F}$, it follows that $N$ is minimal within $\text{Reali}_{M}(F)$.
\end{proof}

For each $F\in \Pi_{X}$, let $M_{F}$ denote the unique matroid in Lemma~\ref{unico}. 

\begin{example}
Consider the configuration $M_{\text{Fano}}$ from Example~\ref{exi}. For the formula $F$ from Example~\ref{exi 2}, the matroid $N$  from Example~\ref{exi 3} is the unique minimal matroid in $\text{Reali}_{M_{\text{Fano}}}(F)$.
\end{example}

We now present an algorithm to identify $\min\{\mathcal{A}\}$ using a {\em Depth-First Search (DFS)} approach. Recall the definition of a {\em stack}, which is an abstract data type that represents a collection of elements with two operations: {\em push}, which adds an element to the collection and {\em pop} which removes the most recently added element.

\begin{algorithm}\label{algo 2}
Algorithm for computing $\min\{\mathcal{A}\}$ for a point-line configuration $M$ with lines $\mathcal{L}$.

\smallskip
\noindent{\bf Initialization:} 
\noindent Initialize a stack $L$ and push the formulas $(1\not \sim 2)$ and $(1\sim 2)$ in that order. Additionally, create an empty list $Y$.

\smallskip
\noindent{\bf Exploration:} While the stack $L$ is not empty:
\begin{itemize}
\item Visit the top formula of $L$, denote it by $F$.
\item Compute $(\mathcal{L})_{F}$, %as described in Remark~\ref{recursive}, 
using the stored information from $(\mathcal{L})_{F\rq}$ for any previously visited formula $F\rq$ that is immediately smaller than $F$. Store this information for future reference.
\item Pop $F$ from the stack.
\item If $F\not \in \Pi_{X}$, select two %distinct 
elements $x,y\in [d]$ such that there exist %distinct edges 
$e_{1},e_{2}\in (\mathcal{L})_{F}$ satisfying
$\{\overline{x}_{F},\overline{y}_{F}\}\subset e_{1}\cap e_{2}$.
Push onto $L$ the formulas 
$F\wedge (x \not \sim y)$, and $F\wedge (x\sim y)$,
in that order, if they were not already present in $L$ prior this step.
\item If $F\in \Pi_{X}$, add the matroid $M_{F}$ to %the list 
$Y$. %only if there is no matroid $N\in Y$ such that $M_{F}\geq N$. 
\end{itemize}

\smallskip
\noindent
{\bf Termination:} The algorithm concludes when $L$ is empty. At this point compare all matroids in the set $Y$ to identify minimal ones and denote the resulting set $Z$. Finally, we have $Z=\min\{\mathcal{A}\}$. 
\end{algorithm}

We now prove the correctness of Algorithm~\ref{algo 2}.

\begin{theorem}
   On input a point-line configuration $M$ with lines $\mathcal{L}$, Algorithm~\ref{algo 2} terminates and outputs $\min\{\mathcal{A}\}$, where $\mathcal{A}=\{N>M: \mathcal{C}_{1}(N)=\emptyset\ \text{and}\ \mathcal{C}_{2}(N)\neq\emptyset\}$.
\end{theorem}
%\medskip
\begin{proof}
To show that the algorithm achieves the expected outcome, we establish the following claims:

\medskip
{\bf Claim~1.} Suppose that, at some step, we encounter $F\in \Pi\backslash \Pi_{X}$.  Let $N$ be any matroid in $\text{Reali}_{M}(F)$. Then, there exists a formula of the form $F\wedge a$ that will be visited in the future, with $N\in \text{Reali}_{M}(F\wedge a)$.

\medskip
By the description of the algorithm, we know that there exist $x,y\in [d]$ such that the formulas
$F\wedge (x\sim y)$ and $F\wedge (x\not \sim y)$,
will be visited. Since $N\in \text{Reali}_{M}(F)$, we have
\[N\in \text{Reali}_{M}(F\wedge (x\sim y))\cup \text{Reali}_{M}(F\wedge (x\not \sim y)), \]
which proves the claim. 

\medskip
{\bf Claim~2.} Let $N\in \mathcal{A}$. Then, upon completion of the algorithm, there exists an element $N\rq\in Z$ such that $N\geq N\rq$.

\medskip
Initially, we have $N\in \textup{Reali}_{M}(1\sim 2)$ or  $N\in\textup{Reali}_{M}(1 \not \sim 2)$. Consequently, by Claim~$1$, we know that we will eventually visit a formula $F\in \Pi_{X}$ with $N\in \text{Reali}_{M}(F)$. By Lemma~\ref{unico}, this implies $N\geq M_{F}$. Since the matroid $M_{F}$ is greater or equal than some element of $Z$, the claim follows. 

\medskip
This proves the correctness of the algorithm.
\end{proof}

\begin{example}\label{ex:quad}
To illustrate how Algorithm~\ref{algo 2} works, we apply its first steps for the point-line configuration $\text{QS}$ depicted in Figure~\ref{fig:combined} (Right), where $\mathcal{L}$ represents its set of lines. %The algorithm proceeds as follows:
\begin{itemize}
\item We initialize a stack $L$ pushing the formulas $(1\sim 2)$ and $(1\not \sim 2)$. Additionally, we create an empty list $Y$.
\item We compute $(\mathcal{L})_{(1\sim 2)}=\{\{1,5,6\},\{1,4,6\},\{3,4,5\}\}$. Since $\{1,6\}\subset \{1,5,6\}\cap \{1,4,6\}$ we push the formulas $(1\sim 2)\wedge (1\not \sim 6)$ and $(1\sim 2)\wedge (1 \sim 6)$.
\item We compute $(\mathcal{L})_{(1\sim 2)\wedge (1\sim 6)}=\{\{3,4,5\}\}.$ Since $(1\sim 2)\wedge (1\sim 6)\in \Pi_{X}$ we add the matroid $\text{QS}_{(1\sim 2)\wedge (1\sim 6)}$ to $Y$. %This matroid is depicted in Figure~\ref{new figure 4} (Left) and 
We denote this matroid by $N_{1}$ (Figure~\ref{new figure 4} (Left)).
\item We compute $(\mathcal{L})_{(1\sim 2)\wedge (1 \not \sim 6)}=\{\{1,4,5,6\},\{3,4,5\}\}.$ Since both edges intersect at $\{4,5\}$, we push the formulas $(1\sim 2)\wedge (1 \not \sim 6)\wedge (4\not \sim 5) $ and $(1\sim 2)\wedge (1 \not \sim 6)\wedge (4\sim 5)$. 
\item We compute $(\mathcal{L})_{(1\sim 2)\wedge (1 \not \sim 6)\wedge (4\sim 5)}=\{\{1,4,6\}\}.$ Since $(1\sim 2)\wedge (1 \not \sim 6)\wedge (4\sim 5)\in \Pi_{X}$ we add the matroid $\text{QS}_{(1\sim 2)\wedge (1 \not \sim 6)\wedge (4\sim 5)}$ to $Y$. %This matroid is depicted in Figure~\ref{new figure 4} (Center) and 
We denote this matroid by $N_{2}$ (Figure~\ref{new figure 4} (Center)).
\item We compute $(\mathcal{L})_{(1\sim 2)\wedge (1 \not \sim 6)\wedge (4\not \sim 5)}=\{\{1,3,4,5,6\}\}.$ Since $(1\sim 2)\wedge (1 \not \sim 6)\wedge (4\not \sim 5)\in \Pi_{X}$ we add the matroid $\text{QS}_{(1\sim 2)\wedge (1 \not \sim 6)\wedge (4\not \sim 5)}$ to $Y$. %This is the matroid depicted in Figure~\ref{new figure 4} (Right) and 
We denote this matroid by $N_{3}$ (Figure~\ref{new figure 4} (Right)).
%\item Let $N_{4}$ be the matroid $\text{QS}_{(1\sim 4)\wedge (1\not \sim 3)\wedge (1\not \sim 5)}$.
\end{itemize}
Following the algorithm to its completion, we find that $\min \{A\}$ consists of $25$ minimal matroids, each derived from the matroids $N_{i}$ by applying automorphisms of $\text{QS}$. %; see Figure~\ref{new figure 4} from left to right.
\end{example}

\begin{figure}[H]
    \centering
    \includegraphics[width=0.75\textwidth, trim=0 0 0 0, clip]{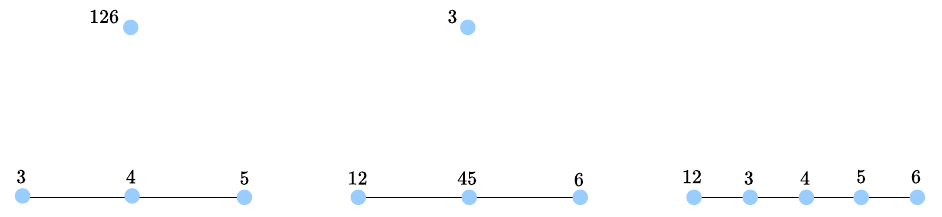}
    \caption{Minimal matroids of the quadrilateral set from Example~\ref{ex:quad}.}
    \label{new figure 4}
\end{figure}

\subsection{Determining $\min\{\mathcal{B}\}$}

To describe $\min\{\mathcal{B}\}$ from Definition~\ref{minimal matroids} and Equation~\eqref{unio}, we introduce the following definition.

\begin{definition}
For each $x\in \textstyle \binom{[d]}{3}\backslash \mathcal{C}_{3}(M)$, we define $M^x$ as the unique point-line configuration on the ground set $[d]$ that is minimal with the property that $M^x\geq \mathcal{L}\cup \{x\}$.
\end{definition}

The existence and uniqueness of $M^{x}$ are guaranteed by Lemma~\ref{unique}. Moreover, as shown in the proof of this lemma, we have 
$\mathcal{L}_{M^{x}}=(\mathcal{L}\cup \{x\})_{F},$
where 
$F=\wedge_{i\neq j\in [d]} (i\not \sim j),$
which can be recursively computed following the steps of Algorithm~\ref{delta p}.

\begin{example}
Consider the point-line configuration $M_{\text{Fano}}$ depicted in Figure~\ref{new figure} (Left). In this case $M_{\text{Fano}}^{x}$, is equal to the uniform matroid $U_{2,7}$ for any $x\in \textstyle \binom{[7]}{3}\setminus \mathcal{C}_{3}(M)$.
\end{example}

\begin{remark}\label{rema}
Let $x\in \textstyle \binom{[d]}{3}$ be such that $\size{x\cap l}\leq 1$ for each $l\in \mathcal{L}$. Then, the set $\mathcal{L}\cup \{x\}$ forms the collection of lines of a point-line configuration, implying that $M^{x}$ is a minimal element of $\mathcal{B}$. 
\end{remark}

We define the following subsets of $\textstyle \binom{[d]}{3}\backslash \mathcal{C}_{3}(M)$:
\begin{equation*}
\mathcal{S}=\{x\in \textstyle \binom{[d]}{3}: \text{$\size{x\cap l}\leq 1$ for each $l\in \mathcal{L}$}\},\quad \text{and} \quad  \mathcal{T}=(\textstyle \binom{[d]}{3}\setminus \mathcal{C}_{3}(M))\setminus \mathcal{S}.
\end{equation*}
Using Remark~\ref{rema}, it follows that
\begin{equation}\label{mini B}
\min\{\mathcal{B}\}=\{M^{x}:x\in \mathcal{S}\} \cup \min\{\min\{M^{x}:x\in \mathcal{T}\}\cup \{M^{x}:x\in \mathcal{S}\}\}.
\end{equation}

The following lemma is essential for outlining the algorithm to identify $\min \{\mathcal{B}\}$.

\begin{lemma}\label{seq}
Let $x,x_{1},\ldots,x_{n}\in \textstyle \binom{[d]}{3}\backslash \mathcal{C}_{3}(M)$. The following statements hold:
\begin{itemize}
\item[{\rm (i)}] %Let $x\in \textstyle \binom{[d]}{3}\backslash \mathcal{C}_{3}(M)$ and 
Let $N\geq M$ be a point-line configuration. Then $N\geq M^{x}$ if and only if $x\in \mathcal{C}_{3}(N)$.
\item[{\rm (ii)}] %Let $x_{1},\ldots,x_{n}$ be a sequence of subsets in $\textstyle \binom{[d]}{3}\backslash \mathcal{C}_{3}(M)$ 
Suppose that $x_{i+1}\in \mathcal{C}_{3}(M^{x_{i}})$ for each $i\in [n]$, identifying $1=n+1$. Additionally, assume there is no $y\not \in \{x_{1},\ldots,x_{n}\}$ such that $y\in \mathcal{C}_{3}(M^{x_{n}})$. Then, $M^{x_{n}}\in \min\{\mathcal{B}\}$.
\end{itemize}
\end{lemma}

\begin{proof}
{\rm (i)} This follows directly from the definition of $M^{x}$. 

\medskip
{\rm (ii)} By item (i), we have $M^{x_{i+1}}\geq M^{x_{i}}$ for all $i\in [n]$, which implies:
$M^{x_{1}}=\cdots=M^{x_{n}}$.
Furthermore, since no $y\not \in \{x_{1},\ldots,x_{n}\}$ belongs to 
$\mathcal{C}_{3}(M^{x_{n}})$, it follows that: \[\mathcal{C}_{3}(M^{x_{n}})=\mathcal{C}_{3}(M)\cup \{x_{1},\ldots,x_{n}\}.\]
To prove that $M^{x_{n}}$ is minimal, assume there exists a point-line configuration $N$ with $M<N<M^{x_{n}}$. Then, for some $i\in [d]$ we must have that $x_{i}\in \mathcal{C}_{3}(N)$. By item (i), this implies $N\geq M^{x_{i}}=M^{x_{n}}$, contradicting the assumption.
\end{proof}

We now introduce an algorithm for identifying $\min(\mathcal{B})$.

\begin{algorithm}\label{algo B}
Algorithm for computing $\min\{\mathcal{B}\}$ for a point-line configuration $M$ with lines $\mathcal{L}$.

\smallskip
{\bf Initialization:} Create a list $Y$ to track the elements already visited, and add all elements from $\mathcal{S}$. Next, create a list $L$ and insert an arbitrary element $x\in \mathcal{T}$. Finally, construct a list $K$ containing the matroids $\{M^{x}:x\in \mathcal{S}\}$.

\smallskip
{\bf Exploration:} While $Y$ is a proper subset of $\textstyle \binom{[d]}{3} \backslash \mathcal{C}_{3}(M)$:

\smallskip
If $L$ is not empty:
\begin{itemize}
\item Let $x$ be the element at the tail of $L$, belonging to $\mathcal{T}$. Execute Algorithm~\ref{delta p} for the hypergraph $\mathcal{L} \cup {x}$
%Let $x$ denote the element at the tail of $L$, which belongs to $\mathcal{T}$. Execute Algorithm~\ref{delta p} for the hypergraph $\mathcal{L}\cup \{x\}$ 
and the formula $F=\wedge_{i\neq j\in d} (i\not \sim j),$
until one of the following scenarios occurs:
\begin{itemize}
\item {\bf Scenario~1.} An edge containing an element $y\in Y$ is added to the hypergraph. In this case, remove all elements from $L$ and append them to $Y$.
\item {\bf Scenario~2.} An edge containing an element \[y\in (\textstyle \binom{[d]}{3} \backslash \mathcal{C}_{3}(M))\backslash (Y\cup L)\]  
is added to the hypergraph. In this case, append $y$ to the tail of $L$.
\end{itemize}

\smallskip
If the execution of Algorithm~\ref{delta p} concludes and there is no element $y\in \textstyle (\binom{[d]}{3} \backslash \mathcal{C}_{3}(M))\backslash L,$
such that $y\in \mathcal{C}_{3}(M^{x})$, then remove all elements from $L$, appending them to $Y$ and add $M^{x}$ to %the list 
$K$.
\end{itemize}

If $L$ is empty, add an arbitrary element $x\in \mathcal{T}\backslash Y$.

\medskip
{\bf Termination:} The algorithm concludes when $Y=\textstyle \binom{[d]}{3} \backslash \mathcal{C}_{3}(M)$. Upon completion, it holds that 
$K=\min\{\mathcal{B}\}$.

\end{algorithm}

\begin{theorem}
On input a point-line configuration $M$ with lines $\mathcal{L}$, Algorithm~\ref{algo B} terminates and outputs $\min\{\mathcal{B}\}$, where $\mathcal{B}=\{N>M: \mathcal{C}_{1}(N)=\emptyset\ \text{and}\ \mathcal{C}_{2}(N)=\emptyset\}$.
\end{theorem}
%\medskip
\begin{proof}
\medskip
We will show that the algorithm produces the expected outcome by proving the following claim: % through induction:

\medskip
{\bf Claim:} After each step, the set $K$ satisfies $K=\min \{M^{z}:z\in Y\}$.

\medskip
We prove the claim by induction. Suppose that, while visiting $x\in \mathcal{T}$, the current state of $K$ is $K=\min \{M^{z}:z\in Y\}$. At this point, let $\{x_{1},\ldots,x_{n}\}$ denote the elements of $L$ with $x_{n}=x$. We aim to show that the equality remains true after executing the algorithm for $x$. We consider three cases:

\medskip
{\bf Case 1.} Suppose that during the execution of Algorithm~\ref{delta p} for the hypergraph $\mathcal{L}_{M}\cup \{x\}$ and the formula $\textstyle{F=\bigwedge_{i\neq j\in d} (i\not \sim j)}$,
we first encounter Scenario~1. In this case, $y\in \mathcal{C}_{3}(M^{x})$. Moreover, we have that $x_{i}\in \mathcal{C}_{3}(M^{x_{i+1}})$ for $i\leq n-1$, which gives
$M^{x_{1}}\geq \cdots \geq M^{x}\geq M^{y}$.
Thus, when the elements of $L$ are removed and appended to $Y$, the equality $K=\min \{M^{z}:z\in Y\}$ still holds.

\medskip
{\bf Case~2.} Suppose we first encounter Scenario~2 during the execution of Algorithm~\ref{delta p}. In this situation, the lists $Y$ and $K$ remain unchanged, so the equality $K=\min \{M^{z}:z\in Y\}$ persists.

\medskip
{\bf Case~3.} Suppose the execution of Algorithm~\ref{delta p} concludes without finding any element \[y\in \textstyle (\binom{[d]}{3} \backslash \mathcal{C}_{3}(M))\backslash L,\]
such that $y\in \mathcal{C}_{3}(M^{x})$. Let $j\leq n-1$ be the smallest integer such that $x_{j}\in \mathcal{C}_{3}(M^{x})$. Moreover, since $x_{i}\in \mathcal{C}_{3}(M^{x+1})$ for $i\leq n-1$, we obtain
$M^{x_{1}}\geq \cdots \geq M^{x}\geq M^{x_{j}}$,
which further implies that
\[M^{x_{j}}=\cdots=M^{x}.\]
Since no element $y\not \in \{x_{j},\ldots,x\}$ satisfies $y\in \mathcal{C}_{3}(M^{x})$, by Lemma~\ref{seq}, we have $M^{x}\in \min\{\mathcal{B}\}$. Therefore, when this step is finalized, the elements $\{x_{1},\ldots,x_{n}\}$ are added to $Y$, and the minimal matroid $M^{x}$ is added to $K$, preserving the equality $K=\min \{M^{z}:z\in Y\}$.

\medskip
Upon completion of the algorithm, we have $Y=\textstyle \binom{[d]}{3}\backslash \mathcal{C}_{3}(M)$, and thus $K=\min \{\mathcal{B}\}$, as desired.
\end{proof}
\subsection{Determining $\min \{\mathcal{C}\}$}

To describe $\min\{\mathcal{C}\}$ from Definition~\ref{minimal matroids} and Equation~\eqref{unio}, we introduce the following definition.

\begin{definition}
For each $i\in [d]$, we define $M(i)$ as the matroid obtained by designating $i$ as a loop. The circuits of this matroid are given by
$\mathcal{C}(M(i))=\mathcal{C}(M\backslash i)\cup \{\{i\}\}$.
\end{definition}

It is straightforward to verify that $\min \{\mathcal{C}\}=\{M(i):i\in [d]\}$. We now establish the following fundamental lemma. Recall that $\mathcal{L}$ denotes the set of lines of $M$.

\begin{lemma}\label{cond min}
Let $i\in [d]$. Then $M(i)\in \min(M)$ if and only if the following conditions hold:
\begin{itemize}
\item%[{\rm (i)}] 
For each point $j\neq i$ there exist $l_{1},l_{2}\in \mathcal{L}$ with $i\in l_{1},j\in l_{2}$ and $(l_{1}\cap l_{2})\backslash \{i,j\}\neq \emptyset$. 
\item%[{\rm (ii)}] 
For each pair of points $j,k\neq i$ there exist two points in $\{i,j,k\}$ that lie on a common line.
\item%[{\rm (iii)}] 
For any line $l\in \mathcal{L}$, there exists a point on $l$ that belongs to a common line with $i$.
\end{itemize}
\end{lemma}

\begin{proof}
First, assume there exists a point $j\neq i$ such that there do not exist distinct lines $l_{1},l_{2}\in \mathcal{L}$ with $i\in l_{1},j\in l_{2}$ and $l_{1}\cap l_{2}\backslash \{i,j\}\neq \emptyset$. We can then define a matroid $N$ of rank three with the following set of circuits of size at most three:
\begin{equation*}
\mathcal{C}_{1}(N)=\emptyset, \quad \mathcal{C}_{2}(N)=\{i,j\}, \quad \text{and} \quad   \mathcal{C}_{3}(N)=\{c\in \mathcal{C}_{3}(M): \{i,j\}\not \subset c\}
\end{equation*}
Given the condition on $j$, this matroid is well defined. Moreover, it is straightforward to verify that $M<N$ and $N\backslash i=M\backslash i$, implying that $N<M(i)$. Therefore, in this case, $M(i)$ would not be minimal. 

Next, suppose there are distinct points $j,k\neq i$ such that no two points in $\{i,j,k\}$ lie on a common line. Here, we define the point-line configuration $N$ with lines given by 
$\mathcal{L}_{N}=\mathcal{L}\cup \{\{i,j,k\}\}$.
Given the condition on $j$ and $k$, this point-line configuration is well defined. It is easy to verify that $M<N$ and $N\backslash i=M\backslash i$, implying that $N<M(i)$. Therefore, in this case, $M(i)$ would not be minimal.

Finally, assume there exists a line $l\in \mathcal{L}$ such that $i$ does not lie on a common line with any point in $l$. In this case, we define the point-line configuration $N$ with the following set of lines:
\[\mathcal{L}_{N}=(\mathcal{L}\cup \{l\cup \{i\}\})\backslash \{l\}\]
Given the condition on $l$, this point-line configuration is well defined. It is easy to verify that $M<N$ and $N\backslash i=M\backslash i$, implying that $N<M(i)$. Therefore, in this case, $M(i)$ would not be minimal.

Now assume that $i\in [d]$ satisfies the three conditions. To prove that $M(i)$ is minimal suppose for contradiction that there exists a matroid $M<N<M(i)$. We will separate this into cases: %Since $M<N$ there exists $c\in \mathcal{C}(N)\backslash \mathcal{D}(M)$.
%We separate into cases for $c$:

\medskip
{\bf Case~1.} $N$ has a loop $\{j\}$ for some $j\in [d]$. 

\medskip
{\bf Case~1.1.} If $j=i$, then $N\geq M(i)$, which is a contradiction.

\medskip
{\bf Case~1.2.} If $j\neq i$, then $N\not<M(i)$, which is a contradiction.
%this contradicts the fact that $i$ is the only loop in $M(i)$.

\medskip
{\bf Case~2.} $N$ has no loops but contains a double point $\{k,j\}$. Since $\{k,j\}\in \mathcal{D}(M(i))$, we must have $i\in \{j,k\}$. Without loss of generality, assume $k=i$. By condition (i), there exist distinct lines $l_{1},l_{2}\in \mathcal{L}$ with $i\in l_{1},j\in l_{2}$ and $l_{1}\cap l_{2}=\{r\}$.

\medskip
{\bf Case~2.1.} If $\{i,r\}\not \in \mathcal{C}(N)$, then $l_{1}\cup l_{2}\subset \closure{\{i,r\}}$ in $N$, implying that $\rank_{N}(l_{1}\cup l_{2})\leq 2$. This contradicts $\rank_{M(i)}(l_{1}\cup l_{2})= 3$.

\medskip
{\bf Case~2.2.} If $\{i,r\}\in \mathcal{C}(N)$, then $\{j,r\}\in \mathcal{C}(N)$, contradicting $\{j,r\}\not \in \mathcal{C}(M(i))$.

\medskip
{\bf Case~3.} Suppose $N$ has no loops and no double points. Since $M<N$, there exists a set $\{j,k,s\}\in \mathcal{C}(N)$ that is independent in $M$. Since $N<M(i)$, it follows that $\{j,k,s\}$ is dependent in $M(i)$, implying $i\in \{j,k,s\}$. Without loss of generality, assume $i=s$. 

\medskip
{\bf Case~3.1.} Suppose that $j$ and $k$ do not belong to a common line. By condition (ii), the point $i$ lies on a common line with either $j$ or $k$. Suppose, without lost of generality, that this point is $j$, and let $l$ denote this line. In $N$, we have $\rank_{N}(\{i,j,k\})=2$ and $\rank(l)=2$, and since $\{i,j\}\in \{i,j,k\}\cap l$, it follows that $\rank_{N}(l\cup \{j,k\})\leq 2$, which contradicts $\rank_{M(i)}(l\cup \{j,k\})=3$.

\medskip
{\bf Case~3.2} Suppose instead that $j$ and $k$ belong to a common line $l\rq \in \mathcal{L}$. 
By condition (iii), there exists a point $r\in l\rq$ and a line $l\in \mathcal{L}$ such that $\{r,i\}\subset  l$. In $N$, we have $\rank_{N}(\{i\}\cup l\rq)=2$ and $\rank(l)=2$, and since $\{i,r\}\in (\{i\}\cup l\rq) \cap l$, it follows that $\rank_{N}(l\cup \{j,k\})\leq 2$, again contradicting $\rank_{M(i)}(l\cup \{j,k\})=3$.

\medskip
This completes the proof.
\end{proof}

\begin{example}
Consider the point-line configurations $M_{\text{Fano}}$ and $M_{\text{Pappus}}$ depicted in Figure~\ref{new figure} (Left) and Figure~\ref{new figure} (Center), respectively.
\begin{itemize}
\item Each point of $M_{\text{Fano}}$ meets the three conditions of Lemma~\ref{cond min}, leading to $M_{\text{Fano}}[0]=[7]$. 
\item For $M_{\text{Pappus}}$, the triples $\{1,4,9\},\{3,6,7\}$ and $\{2,5,8\}$ satisfy that no pair of points within any triple lies on a common line of $M_{\text{Pappus}}$. Hence, we conclude that $M_{\text{Pappus}}[0]=\emptyset$,  since no point fulfills the second condition of Lemma~\ref{cond min}.
\end{itemize}
\end{example}

\begin{definition}\label{M[0]}
    Let $M[0]$ be the set of points satisfying conditions in Lemma~\ref{cond min}. By this lemma, %we have
\begin{equation}\label{condi}\min(M)\cap \mathcal{C}=\{M(i):i \in M[0]\}.\end{equation}
\end{definition}

We now introduce an algorithm for identifying $M[0]$ from Definition~\ref{M[0]}.

\begin{algorithm}\label{algo c}
{\rm Algorithm for computing $M[0]$ for a point-line configuration $M$, with lines $\mathcal{L}$. 
%It follows these steps:
\begin{itemize}
    \item For %any 
    every pair of lines $l_{1},l_{2}\in \mathcal{L}$ that intersect, draw an edge between each point in $l_{1}\backslash (l_{1}\cap l_{2})$ and each point in $l_{2}\backslash (l_{1}\cap l_{2})$. Let $G_{1}$ be the resulting graph on $[d]$. 
 \item Define $C_{1}$ as the set of all points in $G_{1}$ that have degree $d-1$.

\item For %each 
every point $i\in [d]$ and every point  $j\in [d]\backslash \cup_{l\in \mathcal{L}_{i}}l$, draw an edge between $i$ and $j$. Denote by $G_{2}$ the graph on the vertex set $[d]$ obtained in this way, where two points are connected if and only if they do not belong to a common line. 

\item Let $C_{2}$ be the set of points in $G_{2}$ that do not form part of any $3$-clique.

\item For each
 point $j\in C_{1}\cap C_{2}$ verify whether every line $l\rq \not \in \mathcal{L}_{j}$ has non trivial intersection with the set $\cup_{l\in \mathcal{L}_{j}}l$. Define $C_{3}$ as the points satisfying this condition.
\end{itemize}

\noindent{\bf Termination:} We have $C_{3}=M[0]$.
}\end{algorithm}

\begin{theorem}
  On input a point-line configuration $M$ with lines $\mathcal{L}$, 
   Algorithm~\ref{algo c} terminates and 
   outputs $M[0]$ from Definition~\ref{M[0]}.
\end{theorem}
\begin{proof}
To show that the algorithm achieves the expected outcome, we establish the following claims:

\medskip
{\bf Claim~1.} Any point $i\not \in C_{1}$ does not satisfy condition (i).

\medskip
Since $i\not \in C_{1}$, there exists a point $j\neq i$ such that $(i,j)$ is not an edge of $G_{1}$, meaning that there are no distinct lines $l_{1},l_{2}\in \mathcal{L}$ with $i\in l_{1},j\in l_{2}$ and $(l_{1}\cap l_{2})\backslash \{i,j\}\neq \emptyset$.

\medskip
{\bf Claim~2.} Any point $i\not \in C_{2}$ does not satisfy condition (ii).

\medskip
Since $i\not \in C_{2}$, there exist distinct points $j,k\neq i$ such that no two points in $\{i,j,k\}$ lie on a common line. This implies that $i$ does not satisfy the second condition.

\medskip
From these claims, we conclude that $M[0]\subset C_{1}\cap C_{2}$. The correctness follows from the next claim.

\medskip
{\bf Claim~3.} A point $i\in C_{1}\cap C_{2}$ belongs to $M[0]$ if and only if $i\in C_{3}$.

\medskip
Since $i\in C_{1}$, it clearly satisfies condition (i). Additionally, since $i\in C_{2}$, it fulfills condition (ii). 
Therefore, it belongs to $M[0]$ if and only if it satisfies condition (iii). This is precisely the criterion for a point to belong to $C_{3}$. Thus, the correctness of the algorithm is established.
\end{proof}

\subsection{Algorithm for identifying $\min(M)$}

We are now ready to present our main algorithm for determining $\min(M)$.

\begin{algorithm}\label{algo m}
For a point-line configuration $M$, the algorithm proceeds with the following steps:

\begin{itemize}
    
\item Identify $\min \{\mathcal{A}\}$ using Algorithm~\ref{algo 2}.

\item  Identify $\min\{\mathcal{B}\}$ using Algorithm~\ref{algo B}.

\item  Identify $M[0]$ using Algorithm~\ref{algo c}.

\item  For each $N\in \min\{\mathcal{A}\}$, check if there is no $N\rq \in \min \{\mathcal{B}\}$ such that $N\geq N\rq$. Let $L$ be the set of matroids in $\min\{\mathcal{A}\}$ that satisfy this condition.
\end{itemize}
{\bf Termination:} We have $\min(M)=L\cup \min \{\mathcal{B}\} \cup \{M(i): i\in M[0]\}$.
\end{algorithm}

\begin{theorem}
On input a point-line configuration $M$, Algorithm~\ref{algo m} terminates and outputs $\min(M)$, the set of minimal matroids of $M$. % in Definition~\ref{M[0]}.
\end{theorem}
\begin{proof}
By Equation~\eqref{unio}, we know that
$\min(M)=\min\{\min\{\mathcal{A}\}\cup \min\{\mathcal{B}\}\cup \min\{\mathcal{C}\}\}$.
By Lemma~\ref{cond min}, we establish that $\mathcal{C}\cap \min(M)=\{M(i):i\in M[0]\}$. Moreover, since no matroid in $\mathcal{B}$ has loops or double points, none of them can be greater than or equal to a matroid in $\mathcal{A}$ or $\mathcal{C}$, implying that $\min\{\mathcal{B}\}\subset \min(M)$. Similarly, no matroid in $\mathcal{A}$ contains a loop, ensuring that none of them is greater than or equal to a matroid in $\mathcal{C}$. Consequently, we conclude that
$\min(M)=L\cup \min \{\mathcal{B}\} \cup \{M(i): i\in M[0]\}$,
which completes the proof.
\end{proof}

\section{Decomposition strategy}\label{strat}

We now propose a strategy to determine the irreducible decomposition of circuit varieties fore point-line configurations. The following result plays a key role.

\begin{proposition}\label{deco circ}
Let $M$ be a matroid. Then
$V_{\mathcal{C}(M)}=\bigcup_{N\in \min(M)}V_{\mathcal{C}(N)}\ \cup \ V_{M}$.
\end{proposition}

\begin{proof}
The inclusion $\supset$ is clear. To establish the other inclusion, let $\gamma$ be a collection of vectors in $V_{C(M)}$. The collection $\gamma$ defines a matroid $N(\gamma)\geq M$, for which $\gamma \in \Gamma_{N(\gamma)}$. We consider two cases.

\medskip
{\bf Case~1.} If $N(\gamma)=M$, then $\gamma \in V_{M}$.

\medskip
{\bf Case~2.} If $N(\gamma)>M$, then there exists $N\in \min(M)$ such that $N(\gamma)>N$, implying $\gamma \in V_{\mathcal{C}(N)}$.

\medskip
Thus, the other inclusion holds, completing the proof.
\end{proof}

\medskip
\noindent{}
{\bf Decomposition Strategy.} For a point-line configuration $M$ the strategy proceeds as follows: 

\begin{itemize}
\item We begin by identifying the minimal matroids $\min(M)$ using Algorithm~\ref{algo m}.

\item Next, the circuit variety is decomposed using Proposition~\ref{deco circ}.

\item Apply the above two steps %(1) and (2) 
to each circuit variety that appears in the equation in Proposition~\ref{deco circ}. Continue this process iteratively for any new circuit varieties encountered, until each circuit variety falls into one of the following cases:

\medskip
{\bf Case~1.} If a circuit variety $V_{\mathcal{C}(N\rq)}$ corresponds to a nilpotent point-line configuration with no points of degree greater than two, it is replaced by $V_{N\rq}$. (See Theorem~\ref{nil coincide}(i)).

\medskip
{\bf Case~2.} If a circuit variety $V_{\mathcal{C}(N\rq)}$ corresponds to a configuration whose points have all degree at most two, and all its proper submatroids are nilpotent, it is replaced by  $V_{N\rq}\cup V_{U_{2,d}}$. (See Theorem~\ref{nil coincide}(ii)).

\medskip 
Note that in the previous cases we view $N$ as a point-line configuration by removing its loops and identifying its double points. After this step, we obtain a decomposition of $V_{\mathcal{C}(M)}$ as a union of matroid varieties.

\item For each matroid variety obtained at the end of step (3), establish its irreducibility by identifying the solvable matroids involved. (See Theorem~\ref{nil coincide}(iii)).

\item Finally, eliminate any redundant component from the decomposition, leaving the irreducible components of $V_{\mathcal{C}(M)}$.

\end{itemize}

\begin{remark}
For a given point-line configuration $M$, one of the main algebraic questions is to determine the defining equations of its associated matroid varieties, i.e., to find a complete set of generators for the associated ideal $I_M$ up to radical. The case where the configuration contains no points of degree greater than two was studied in \cite{Fatemeh4}. However, without this assumption, the problem becomes considerably more difficult. We emphasize that obtaining the irreducible decomposition of $V_{\mathcal{C}(M)}$ could be a crucial step in addressing this challenge. For example, this strategy has proven effective for the Pappus configuration, which will
appear in future work.
\end{remark}

\section{Examples}\label{examples}

In this section, we illustrate our algorithms on 
classical configurations of the form $v_{3}$. We note that the irreducible decompositions of these examples cannot be computed using any existing computer algebra systems. 
For clarity, Theorem~\ref{nil coincide} will be applied throughout this section without explicit mention. The lemmas in this section will be proven later in Section~\ref{appen}.

We have used the implementation of our algorithm to verify the minimal matroids. Algorithm~\ref{algo m}, implemented by Rémi Prébet, is available at the following link: \url{bit.ly/49GYXxC}.
\subsection{Fano plane}

Consider the {\bf Fano plane} $M_{\text{Fano}}$ in Figure~\ref{new figure} (Left). Using Algorithm~\ref{algo m}, we establish that the set 
$\min(M_{\text{Fano}})$ consists of the following matroids; see Figure~\ref{figure min}, from left to right:

\begin{itemize}
\item[{\rm (i)}] The uniform matroid $U_{2,7}$.
\item[{\rm (ii)}] The matroids $M_{\text{Fano}}(i)$ for $i\in [7]$.
%A matroid consisting of a loop together with a quadrilateral set. %The loop could be any of the seven points of $M_{\text{Fano}}$.
\item[{\rm (iii)}] A line of $M_{\text{Fano}}$, with the remaining four points coinciding outside this line. 
\item[{\rm (iv)}] A matroid with one line containing three double points and a free point outside it %There are seven possibilities for this free point. 
\end{itemize}

\begin{figure}
    \centering
    \includegraphics[width=0.8\textwidth, trim=0 0 0 0, clip]{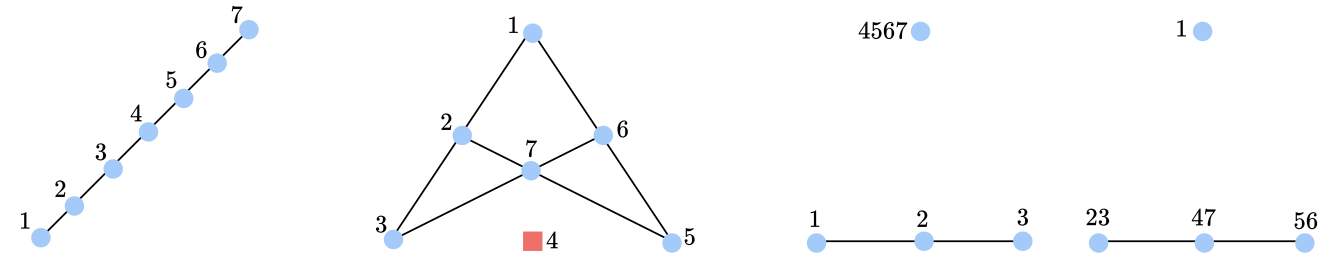}
    \caption{Minimal matroids of Fano configuration.}
    \label{figure min}
\end{figure}

There are seven matroids of the third type, each associated with a line of $M_{\text{Fano}}$, and seven matroids of the fourth type, each determined by the choice of a free point. We label these matroids as
$A_{j},B_{k}$, 
for $j,k\in [7]$. Given that $M_{\text{Fano}}$ is not realizable, this leads to the following decomposition:

\[V_{\mathcal{C}(M_{\text{Fano}})}=V_{\mathcal{C}(U_{2,7})}\bigcup_{i=1}^{7}V_{\mathcal{C}(M_{\text{Fano}}(i))}\bigcup_{j=1}^{7}V_{\mathcal{C}(A_{j})}\bigcup_{k=1}^{7}V_{\mathcal{C}(B_{k})}.\]

Since the matroids $U_{2,7},A_{j}$ and $B_{k}$ are all nilpotent, their
matroid and circuit varieties coincide. Furthermore, by Example~\ref{ej quad}, we have $V_{\mathcal{C}(M_{\text{Fano}}(i))}\subset V_{M_{\text{Fano}}(i)}\cup V_{U_{2,7}}$ for $i\in [7]$. Thus, we obtain:
\[V_{\mathcal{C}(M_{\text{Fano}})}=V_{U_{2,7}}\bigcup_{i=1}^{7}V_{M_{\text{Fano}}(i)}\bigcup_{j=1}^{7}V_{A_{j}}\bigcup_{k=1}^{7}V_{B_{k}}.\]
All matroids in this decomposition are solvable. Thus, %by Theorem~\ref{nil coincide}, 
all these varieties are irreducible. Additionally, the decomposition is non-redundant,
 as no two matroids are comparable with the dependency order. Therefore, the above decomposition is the irreducible decomposition of $V_{\mathcal{C}(M_{\text{Fano}})}$.

\begin{figure}[H]
    \centering
    \includegraphics[width=0.8\textwidth, trim=0 0 0 0, clip]{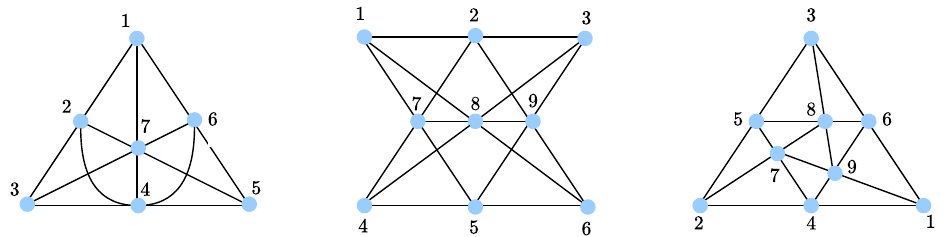}
    \caption{(Left) Fano plane; (Center) Pappus configuration; (Right) Second configuration $9_{3}$}
    \label{new figure}
\end{figure}

\subsection{MacLane configuration}
Consider the {\bf MacLane configuration} $M$ depicted in Figure~\ref{new figure 2} (Left), which is the unique $8_{3}$ configuration. 
Using Algorithm~\ref{algo m}, we find that $\min(M)$ consists of the following matroids; see Figure~\ref{min mac}, from left to right:
\begin{itemize}
\item[{\rm (i)}] The uniform matroid $U_{2,8}$. 
%in which all the points lie on the same line.
\item[{\rm (ii)}] The matroids $M(i)$ for $i\in [8]$. %There are $8$ matroids of this type, corresponding with the choice of the loop.
\item[{\rm (iii)}] A line of $M$, with the remaining five points coinciding at a single point outside it. 
\item[{\rm (iv)}] A line containing four points, where three are double points, along with a free point $x$ outside the line. The double points correspond to pairs that lie on a common line with $x$.
\end{itemize}

 \begin{figure}[H]
    \centering
    \includegraphics[width=0.8\textwidth, trim=0 0 0 0, clip]{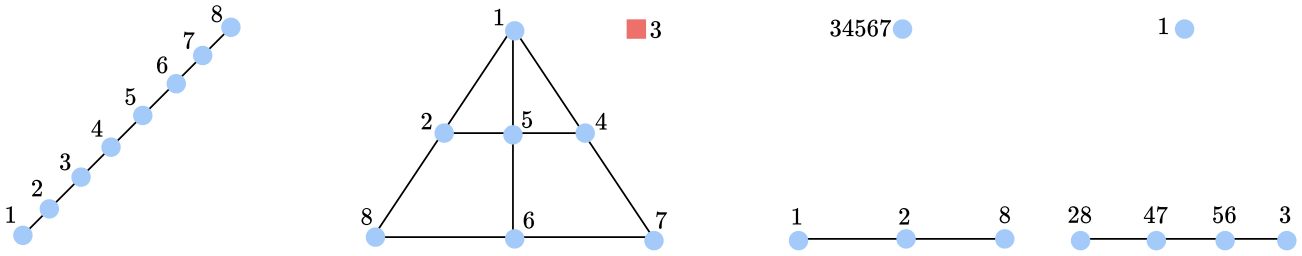}
    \caption{Minimal matroids of MacLane configuration $M$}
    \label{min mac}
\end{figure}

There are eight matroids of the third type, each associated with a line of $M$, and eight matroids of the fourth type, each determined by the choice of a free point. We label these matroids as
$A_{j},B_{k}$, 
for $j,k\in [8]$. This leads to the following decomposition:

\begin{equation}\label{deco mac}
V_{\mathcal{C}(M)}=V_{M}\cup V_{\mathcal{C}(U_{2,8})}\bigcup_{i=1}^{8}V_{\mathcal{C}(M(i))}\bigcup_{j=1}^{8}V_{\mathcal{C}(A_{j})}\bigcup_{k=1}^{8}V_{\mathcal{C}(B_{k})}.
\end{equation}

The following lemma provides the irreducible decomposition of each $V_{\mathcal{C}(M(i))}$.

\begin{lemma}\label{deco N}
Let~$N$~be~the~configuration~with~lines~$\{6,3,4\},\{6,2,7\},\{6,1,5\},\{2,3,5\},\{1,4,7\}$.~Then: 
\begin{equation*}
V_{\mathcal{C}(N)}=V_{N}\cup V_{U_{2,7}}\cup V_{N(6)}.
\end{equation*}
\end{lemma}

Using Lemma~\ref{deco N}, we have $\textstyle{\bigcup_{i=1}^{8}V_{\mathcal{C}(M(i))}=\bigcup_{i=1}^{8}V_{M(i)}\bigcup_{i=1}^{4}V_{D_{i}}\cup\bigcup_{i=1}^{8}V_{C_{i}}}$, where 
\begin{itemize}
\item The matroid $C_{i}$ is obtained by making one point of $M$ a loop, with the remaining points forming the uniform matroid $U_{2,7}$.
\item The matroid $D_{i}$ is obtained
from 
$M$ by making loops of two points that do not belong to a common line. The pairs of non-connected points are $\{1,3\},\{2,4\},\{5,7\},\{6,8\}$. 
\end{itemize}

The matroid $V_{M}$ has two irreducible components, denoted $V_{M}^{+}$ and $V_{M}^{-}$, see \textup{\cite[Table~4.1]{corey2023singular}}. Following the reasoning in \textup{\cite[Example~2.1]{corey2023singular}}, we can explicitly describe these two components. The matroids $U_{2,8},A_{j}$ and $B_{k}$ are all nilpotent, hence their circuit and matroid varieties coincide.
Additionally, note that $V_{C_{i}}\subset V_{U_{2,9}}$ for each $i\in [8]$. Furthermore, we have the following lemma.

\begin{lemma}\label{mov inc}
For every $j\in [8]$, we have $V_{A_{j}}\subset V_{M}$.
\end{lemma} 

Substituting everything into  Equation~\eqref{deco mac}, we obtain:

\begin{equation}\label{maclane}
V_{\mathcal{C}(M)}=V_{M}^{+}\cup V_{M}^{-}\cup V_{U_{2,8}}\bigcup_{i=1}^{8}V_{M(i) }\bigcup_{i=1}^{4}V_{D_{i}}\bigcup_{k=1}^{8}V_{B_{k}}.
\end{equation}

All matroids in this decomposition are solvable, and their varieties are therefore irreducible. Moreover, the decomposition is easily verified as non-redundant irreducible decomposition of $V_{\mathcal{C}(M)}$.

\begin{figure}[H]
    \centering
    \includegraphics[width=0.8\textwidth, trim=0 0 0 0, clip]{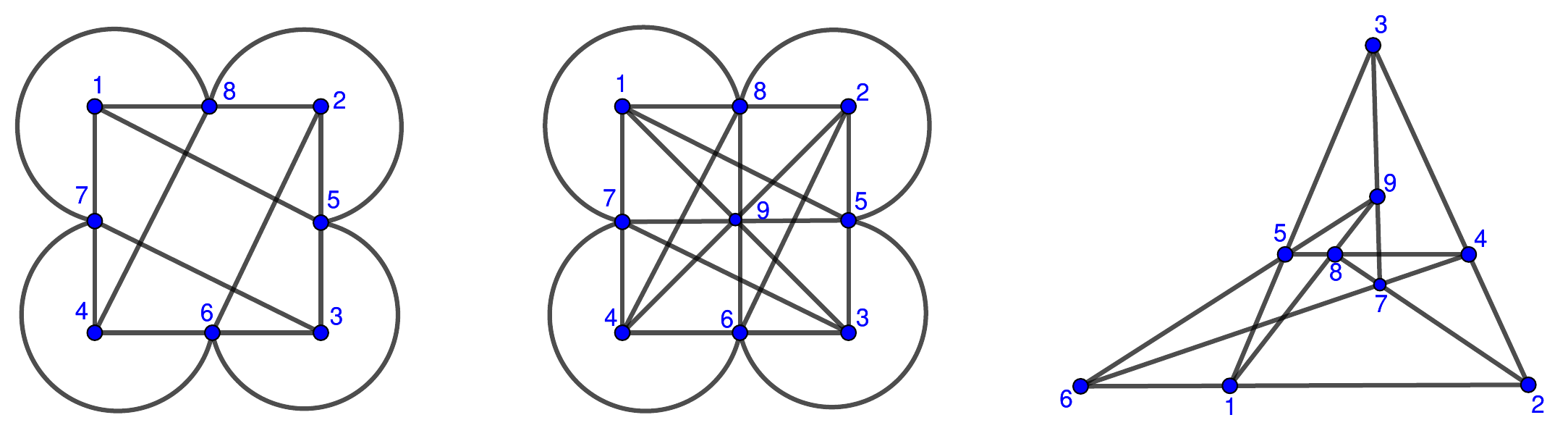}
    \caption{(Left) MacLane configuration; (Center) Affine plane; (Right) Third configuration $9_{3}$}
    \label{new figure 2}
\end{figure}

\subsection{Affine plane of order $3$}
Consider the {\bf Affine plane} $M_{\text{affine}}$ of order three, shown in Figure~\ref{new figure 2} (Center). This configuration represents the affine plane of order $3$ and forms an $S(2,3,9)$ system. Using Algorithm~\ref{algo m}, we determine the minimal matroids of $M_{\text{affine}}$, which are depicted in Figure~\ref{new figure 21}, from left to right:   
\begin{itemize}
\item[{\rm (i)}] The uniform matroid $U_{2,9}$.
%in which all the points lie on the same line.
\item[{\rm (ii)}] %A loop together with the configuration $M_{\text{MacLane}}$. 
The matroids $M_{\text{affine}}(i)$ for $i\in [9]$.
\item[{\rm (iii)}] A line from $M_{\text{affine}}$ with the remaining six points coinciding in a point outside the line. 
\item[{\rm (iv)}] A line containing four points and an additional free point $x$ outside of it. Each of the four points is a double point, with pairs of points that lie on a common line with $x$ coinciding.
\end{itemize}

\begin{figure}[H]
    \centering
    \includegraphics[width=0.8\textwidth, trim=0 0 0 0, clip]{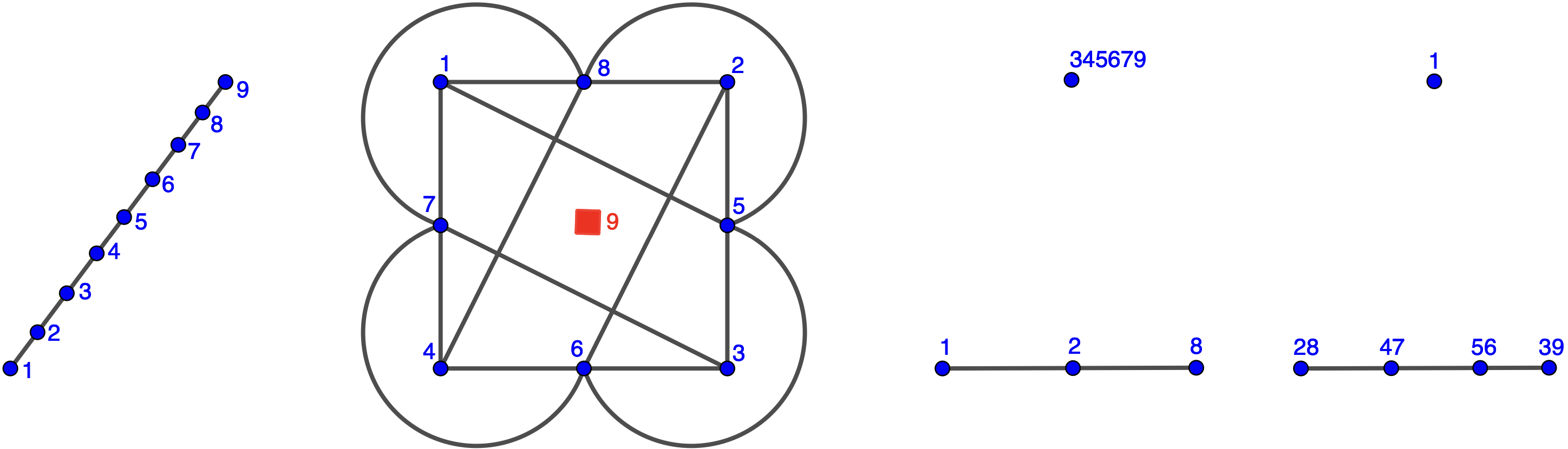}
    \caption{Minimal matroids of $M_{\text{affine}}$}
    \label{new figure 21}
\end{figure}

There are $12$ matroids of the third type, corresponding to the choice of the line, and nine matroids of the fourth type, each determined by the choice of the free point.
We label these matroids as $A_{j},B_{k}$ for $j\in [12]$ and $k\in [9]$. This leads us to the following decomposition:
\begin{equation}\label{deco mac 2}
V_{\mathcal{C}(M_{\text{affine}})}=V_{M_{\text{affine}}}\cup V_{\mathcal{C}(U_{2,9})}\bigcup_{i=1}^{9}V_{\mathcal{C}(M_{\text{affine}}(i))}\bigcup_{j=1}^{12}V_{\mathcal{C}(A_{j})}\bigcup_{k=1}^{9}V_{\mathcal{C}(B_{k})}.
\end{equation}

We observe the following facts, which lead to additional matroids appearing in the decomposition.
\begin{itemize}
\item Since each $M_{\text{affine}}(i)$ consists of a loop alongside the MacLane configuration, the variety $V_{M_{\text{affine}}(i)}$
 likewise has two irreducible components, denoted $V_{i}^{+}$ and $V_{i}^{-}$.
\item For any distinct $i,j\in [9]$, let $D_{i,j}$ denote the matroid obtained from $M_{\text{affine}}$ by making the points $i$ and $j$ loops.

\item Let $E_{i}$, with $i\in [12]$, denote the matroids  obtained from $M_{\text{affine}}$ by making the three points of one of its lines loops.
\item The variety $V_{M_{\text{affine}}}$ has two irreducible components, $V_{\text{affine}}^{+}$ and $V_{\text{affine}}^{-}$, as shown in \textup{\cite[Table~4.1]{corey2023singular}}. 
%matroids are redundant in Equation~\eqref{deco mac 2}.
\end{itemize}

Using the isomorphism between $V_{\mathcal{C}(M_{\text{affine}}(i))}$ and $V_{\mathcal{C}(M_{\text{MacLane}})}$, applying Equation~\eqref{maclane} we obtain:

\begin{equation}\label{maclane 2}
\bigcup_{i=1}^{9}V_{\mathcal{C}(M_{\text{affine}}(i))}\subset V_{U_{2,9}}\bigcup_{i=1}^{9}V_{i}^{+}\bigcup_{i=1}^{9}V_{i}^{-} \bigcup_{i\neq j}V_{D_{i,j}}\bigcup_{j=1}^{12}V_{E_{j}}\bigcup_{k=1}^{9}V_{\mathcal{C}(B_{k})}%\bigcup_{k}V_{F_{k}}.
\end{equation}

\begin{lemma}\label{mov inc 2}
We have $V_{A_{j}}\subset V_{M_{\text{affine}}}$ for all $j\in [12]$.
\end{lemma}

We know that the circuit variety and the matroid variety coincide for the matroids $U_{2,9},A_{j},B_{k}$ and that $V_{A_{j}}\subset V_{M_{\text{affine}}}$. From this, and incorporating Equation~\eqref{maclane 2} within Equation~\eqref{deco mac 2}, we obtain:
\[V_{\mathcal{C}(M_{\text{affine}})}=V_{\text{affine}}^{+}\cup V_{\text{affine}}^{-}\cup V_{U_{2,9}}\bigcup_{i=1}^{9}V_{i}^{+}\bigcup _{i=1}^{9}V_{i}^{-}\bigcup_{i=1}^{12}V_{E_{i}}\bigcup_{i\neq j}V_{D_{i,j}}\bigcup_{k=1}^{9}V_{B_{k}}.\]

All the matroids in this decomposition are solvable, and therefore, their associated varieties are irreducible by Theorem~\ref{nil coincide}. It is also straightforward to verify that this decomposition is non-redundant. Consequently, this is the irreducible decomposition of $V_{\mathcal{C}(M_{\text{affine})}}$.

\subsection{Pappus configuration} Consider the {\bf Pappus configuration} $M_{\text{Pappus}}$ depicted in Figure~\ref{new figure} (Center). This is one of the three $9_{3}$ configurations. Following Algorithm~\ref{algo m}, we find that $\min(M_{\text{Pappus}})$ consists of the following matroids; see Figure~\ref{min pap}, from left to right: 
\begin{itemize}
\item[{\rm (i)}] The three matroids obtained by adding one of the circuits $\{1,4,9\},\{2,5,8\}$, or $\{3,6,7\}$. The one by adding the circuit $\{1,4,9\}$ is depicted in Figure~\ref{min pap} (Left).
\item[{\rm (ii)}] The matroids formed by identifying three points that are pairwise connected by lines of $M_{\text{Pappus}}$ but do not lie on a single line.  
\item[{\rm (iii)}] The matroid derived from identifying $1=3$, $4=6$, and $7=9$, along with all the matroids obtained by applying automorphisms of $M_{\text{Pappus}}$. 
\end{itemize}

\begin{figure}[H]
    \centering
    \includegraphics[width=0.8\textwidth, trim=0 0 0 0, clip]{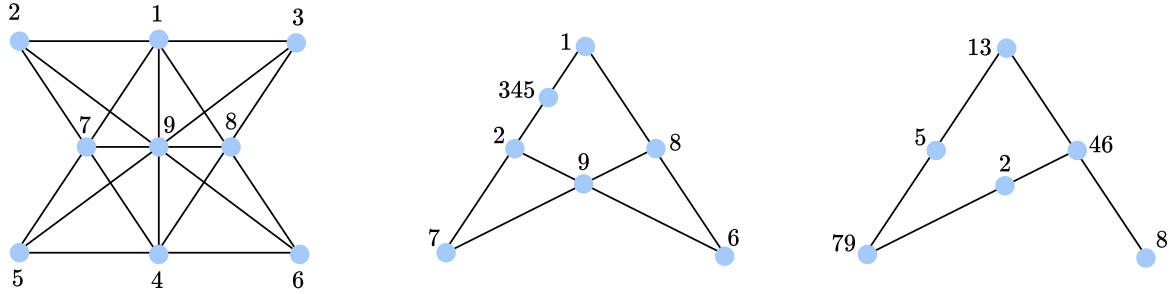}
    \caption{Minimal matroids of Pappus configuration (up to isomorphism)}
    \label{min pap}
\end{figure}

We denote the matroids of the first, second, and third types by $A_{i},B_{j}$ and $C_{k}$, respectively, for $i\in [3],j\in [18],$ and $k\in [9]$. 
From this, we derive the following decomposition:

\begin{equation}\label{equ papus}
V_{\mathcal{C}(M_{\text{Pappus}})}=V_{M_{\text{Pappus}}}\bigcup_{i=1}^{3}V_{\mathcal{C}(A_{i})}\bigcup_{j=1}^{18}V_{\mathcal{C}(B_{j})}\bigcup_{k=1}^{9}V_{\mathcal{C}(C_{k})}.
\end{equation}

Let $D_{i}$, with $i\in [9]$, denote the matroids that consist of a line with five points and a free point $x$ outside of it, where exactly three of the five points are double points, in which pairs of points that belong to a common line with $x$ coincide. We will use the following lemma; see Section~\ref{appen} for its proof. 

\begin{lemma}\label{dl}
The following statements hold:
\begin{itemize}
\item[{\rm (i)}] We have
\begin{equation}\label{ecu dl}
\bigcup_{j=1}^{3}V_{\mathcal{C}(A_{j})}\subset V_{M_{\text{Pappus}}}\bigcup_{i=1}^{9}V_{M_{\text{Pappus}}(i)}\bigcup_{j=1}^{18}V_{\mathcal{C}(B_{j})}\bigcup_{k=1}^{9}V_{\mathcal{C}(C_{k})}\bigcup_{i=1}^{9}V_{D_{i}} \cup V_{U_{2,9}}.
\end{equation}
\item[{\rm (ii)}] For $j\in [18]$, $V_{\mathcal{C}(B_{j})}\subset V_{M_{\text{Pappus}}}\cup V_{U_{2,9}}$.
\item[{\rm (iii)}] For $k\in [9]$, $V_{\mathcal{C}(C_{k})}=V_{C_{k}}\subset V_{M_{\text{Pappus}}}$.
\end{itemize}
\end{lemma}

We know that $V_{M_{\text{Pappus}}}$ is irreducible, as indicated in \textup{\cite[Table~4.1]{corey2023singular}}. Additionally, since the matroids $D_{k}$ are nilpotent, 
their circuit and matroid varieties coincide. 
This, along with Lemma~\ref{dl}, shows that:
\begin{equation}\label{equ papus 2}
V_{\mathcal{C}(M_{\text{Pappus}})}=V_{M_{\text{Pappus}}}\cup V_{U_{2,9}}\bigcup_{i=1}^{9}V_{\mathcal{C}(M_{\text{Pappus}}(i))}\bigcup_{k=1}^{9}V_{D_{k}}.
\end{equation}

To determine the irreducible decomposition of $V_{\mathcal{C}(M_{\text{Pappus}})}$, we must first compute the irreducible decomposition of $V_{\mathcal{C}(M_{\text{Pappus}}(i))}$. %To achieve this, we will first find the minimal matroids of $A_{i}$. 
Since the selections of $i\in [9]$ are all analogous, we assume that our matroid is $M_{\text{Pappus}}(9)$. Following Algorithm~\ref{algo m}, we find that the minimal matroids of $M_{\text{Pappus}}(9)$ that are not redundant in Equation~\eqref{equ papus 2} are:

\begin{itemize}
\item The matroids derived from $M_{\text{Pappus}}(9)$ by making one of the points in $[8]$ a loop.
\item The matroid obtained from $M_{\text{Pappus}}(9)$ by adding one of the circuits $\{2,5,8\}$ or $\{3,6,7\}$.
\end{itemize}

We define the following matroids; see Figure~\ref{figura} from left to right:

\begin{itemize}
\item We denote by $H_{i}$ for $i\in [18]$ the matroids that are obtained from $M_{\text{Pappus}}$ by making one of its points a loop and adding one of the circuits $\{1,4,9\},\{2,5,8\}$ or $\{3,6,7\}$.
\item We denote by $F_{i}$ for $i\in [9]$ the matroids obtained from $M_{\text{Pappus}}$ by turning two points, which are not on the same line, into loops.
\item We denote by $G_{i}$ for $i\in [27]$ the matroids obtained from $M_{\text{Pappus}}$  by turning two points, which are  on the same line, into loops. 
\item We denote $I_{1},I_{2},I_{3}$ as the matroids derived from $M_{\text{Pappus}}$ by turning all three points of one of the triples $\{1,4,9\},\{3,6,7\}$ or $\{2,5,8\}$ into loops.
\end{itemize}

\begin{figure}[H]
    \centering
    \includegraphics[width=0.8\textwidth, trim=0 0 0 0, clip]{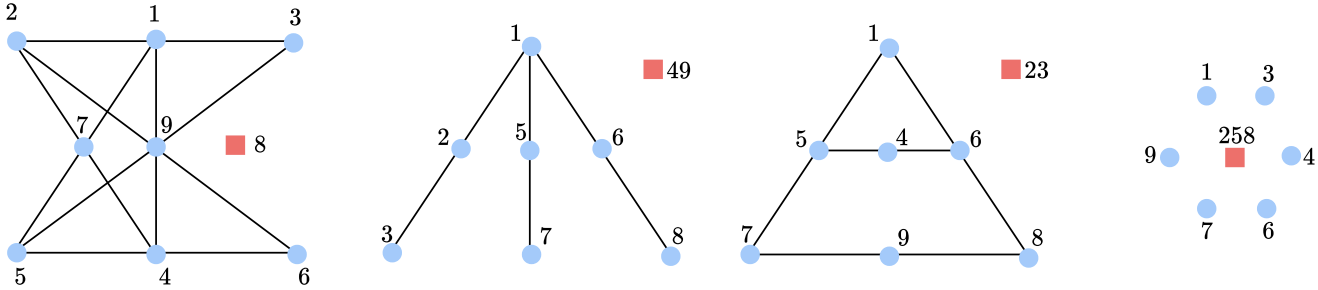}
    \caption{Matroids $H_{i},F_{i},G_{i}$ and $I_{i}$.}
    \label{figura}
\end{figure}

Consequently, from the discussion above on the matroids $M_{\text{Pappus}}(i)$, Equation~\eqref{equ papus 2}
becomes:
\begin{equation}\label{equ papus 3}
V_{\mathcal{C}(M_{\text{Pappus}})}=V_{M_{\text{Pappus}}}\cup V_{U_{2,9}}\bigcup_{i=1}^{9}V_{M_{\text{Pappus}}(i)}\bigcup_{j=1}^{9}V_{\mathcal{C}(F_{j})}\bigcup_{k=1}^{27}V_{\mathcal{C}(G_{k})}\bigcup_{l=1}^{18}V_{\mathcal{C}(H_{l})}\bigcup_{r=1}^{9}V_{D_{r}}.
\end{equation}

We will also apply the following lemma, whose proof will be given in Section~\ref{appen}. %on Page~\ref{page:5.5.}.

\begin{lemma}\label{hl1}
The following statements hold: 
\begin{itemize}
\item[{\rm (i)}] We have
\begin{equation}\label{n ecu}
\bigcup_{l=1}^{18}V_{\mathcal{C}(H_{l})}\subset V_{M_{\text{Pappus}}}\cup V_{U_{2,9}}\bigcup_{i=1}^{9}V_{M_{\text{Pappus}}(i)}\bigcup_{j=1}^{9}V_{\mathcal{C}(F_{j})}\bigcup_{k=1}^{27}V_{\mathcal{C}(G_{k})}\bigcup_{r=1}^{9}V_{D_{r}}.
\end{equation}
\item[{\rm (ii)}] For $k\in [27]$, $V_{\mathcal{C}(G_{k})}\subset V_{M_{\text{Pappus}}}$.
\item[{\rm (iii)}] For $i\in [9]$,  $V_{M_{\text{Pappus}}(i)}\subset V_{M_{\text{Pappus}}}$.
\item[{\rm (iv)}] $V_{\mathcal{C}(F_{j})}=V_{F_{j}}\cup V_{F_{j}\rq}$, where $F_{j}\rq$ is the matroid obtained from $F_{j}$ by making its unique point of degree three a loop.
\end{itemize}
\end{lemma}

From Equation~\eqref{equ papus 3} and applying Lemma~\ref{hl1}, we obtain:
\begin{equation}\label{equ papus 4}
V_{\mathcal{C}(M_{\text{Pappus}})}=V_{M_{\text{Pappus}}}\cup V_{U_{2,9}}\bigcup_{i=1}^{9}V_{F_{i}}\bigcup_{j=1}^{3}V_{I_{j}}\bigcup_{k=1}^{9}V_{D_{k}}.
\end{equation}

All matroids in \eqref{equ papus 4}, except for $M_{\text{Pappus}}$, are solvable, which implies  that their corresponding varieties are irreducible.
Furthermore, it is straightforward to verify that this decomposition is non-redundant. Therefore, we conclude that this is the irreducible decomposition of $V_{\mathcal{C}(M_{\text{Pappus}})}$. 

\subsection{Second configuration $9_{3}$} 
Consider the {\bf second $9_{3}$ configuration} $\mathcal{K}_{9}$ depicted in Figure~\ref{new figure} (Right). Using Algorithm~\ref{algo m}, we determine the minimal matroids of $\mathcal{K}_{9}$ as follows; see Figure~\ref{min K}, from left to right:

\begin{itemize}
\item[{\rm (i)}] The uniform matroid $U_{2,9}$. 
%in which all the points lie on the same line.
\item[{\rm (ii)}] The matroids $\mathcal{K}_{9}(i)$ for each $i\in [9]$.
\item[{\rm (iii)}] A matroid consisting of a line with five points and a free point $x$ outside of the line. Three of the five points are double points, where pairs that share a line with $x$ coincide. 
\item[{\rm (iv)}] The matroid obtained by identifying the points $1,3,4,5,9$, along with all the matroids derived from it by applying an automorphism of $\mathcal{K}_{9}$. 
\item[{\rm (v)}] The matroid obtained by identifying $3$ with $6$ and also  identifying $2,4,5,8$, along with all the matroids derived from it by applying an automorphism of $\mathcal{K}_{9}$. 
\item[{\rm (vi)}] The matroid obtained by identifying $1,4,9$ and $6,8$, as well as the matroids obtained from it by applying an automorphism of $\mathcal{K}_{9}$. 
\item[{\rm (vii)}] The matroid where $1=2=3$ and $\{1,7,8,9\}$ are dependent, and the corresponding matroids under automorphisms of $\mathcal{K}_{9}$. 
\end{itemize}

\begin{figure}[H]
    \centering
    \includegraphics[width=1\textwidth, trim=0 0 0 0, clip]{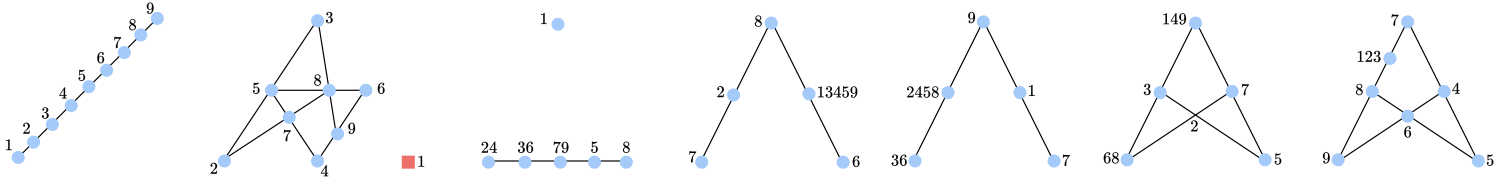}
    \caption{Minimal matroids of the second $9_{3}$ configuration $\mathcal{K}_{9}$ (up to isomorphism)}
    \label{min K}
\end{figure}

Let $A_{i},B_{i},C_{i},D_{i}$ for $i\in [9]$ represent the matroids of types (iii),(iv),(v), and (vi), respectively. Similarly, define $E_{i}$ for $i\in [3]$ as the matroids of type (vii). From this, we obtain the decomposition:

\begin{equation}\label{equ k9}
V_{\mathcal{C}(\mathcal{K}_{9})}=V_{\mathcal{K}_{9}}\cup V_{U_{2,9}}\bigcup_{i=1}^{9}V_{\mathcal{K}_{9}(i)}\bigcup_{i=1}^{9}V_{\mathcal{C}(A_{i})}\bigcup_{i=1}^{9}V_{\mathcal{C}(B_{i})}\bigcup_{i=1}^{9}V_{\mathcal{C}(C_{i})}\bigcup_{i=1}^{9}V_{\mathcal{C}(D_{i})}
\bigcup_{i=1}^{3}V_{\mathcal{C}(E_{i})}.
\end{equation}

We will also use the following lemma: 

\begin{lemma}\label{l1}
The following statements hold:
\begin{itemize}
\item[{\rm (i)}] For $i\in [9]$, $V_{\mathcal{C}(B_{i})}=V_{B_{i}}\subset V_{\mathcal{K}_{9}}$.
\item[{\rm (ii)}] For $i\in [9]$, $V_{\mathcal{C}(C_{i})} =V_{C_{i}}\subset V_{\mathcal{K}_{9}}$. %\cup V_{A_{i}}$.
\item[{\rm (iii)}] For $i\in [9]$, $V_{\mathcal{C}(D_{i})}\subset V_{\mathcal{K}_{9}}\cup V_{U_{2,9}}$.
\item[{\rm (iv)}] For $i\in [3]$, $V_{\mathcal{C}(E_{i})}\subset V_{\mathcal{K}_{9}}\cup V_{U_{2,9}}$.
\end{itemize}
\end{lemma}

Note that the matroids $A_{j}$ are nilpotent for $j\in [9]$, so the matroid and circuit varieties coincide. By applying Lemma~\ref{l1}, we obtain the following:

\begin{equation}\label{equ k9 2}
V_{\mathcal{C}(\mathcal{K}_{9})}=V_{\mathcal{K}_{9}}\cup V_{U_{2,9}}\bigcup_{i=1}^{9}V_{\mathcal{C}(\mathcal{K}_{9}(i))}\bigcup_{j=1}^{9}V_{A_{j}}.
\end{equation}

To determine the irreducible decomposition of $V_{\mathcal{C}(\mathcal{K}_{9})}$ we need to first find the irreducible decomposition of $V_{\mathcal{C}(\mathcal{K}_{9}(i))}$. Since the cases for $i\in [9]$ are analogous, we assume the matroid is $\mathcal{K}_{9}(9)$, and establish the following lemma:

\begin{lemma}\label{l5}
Let $N_{1}$ and $N_{2}$ be the matroids obtained from $\mathcal{K}_{9}(9)$ by setting the points $2$ and $5$ as loops, respectively.
Then, we have that \[V_{\mathcal{C}(\mathcal{K}_{9}(9))}=V_{\mathcal{K}_{9}(9)}\cup V_{N_{1}}\cup V_{N_{2}}\cup V_{U_{2,9}}.\]
\end{lemma}

We denote by $G_{i}$, with $i\in [9]$, the matroids that are obtained from $\mathcal{K}_{9}$ by setting two points that are not on the same line as loops. Using Lemma~\ref{l5}, we obtain that:

\begin{equation}\label{equ k9 3}
V_{\mathcal{C}(\mathcal{K}_{9})}=V_{\mathcal{K}_{9}}\cup V_{U_{2,9}}\bigcup_{i=1}^{9}V_{\mathcal{K}_{9}(i)}\bigcup_{j=1}^{9}V_{A_{j}}\bigcup_{l=1}^{9}V_{G_{l}}.
\end{equation}

We know that $V_{\mathcal{K}_{9}}$ is irreducible, as indicated in \textup{\cite[Table~4.1]{corey2023singular}}. Therefore, since all matroids in this decomposition, except for $\mathcal{K}_{9}$, are solvable, this implies  that the corresponding varieties are irreducible.
Furthermore, it is straightforward to verify that this decomposition is non-redundant. Therefore, we conclude that this is the irreducible decomposition of $V_{\mathcal{C}(\mathcal{K}_{9})}$.

\section{Connections to $\mathcal{X}$-matroids %Further directions
}\label{quest}

We conclude the paper by exploring the connection between our work and $\mathcal{X}$-matroids, as introduced in \cite{jackson2024maximal}. These matroids are defined on $[d]$ such that every subset in a specified family $\mathcal{X}$ forms a circuit. Additionally, we propose a related conjecture.  
The study of $\mathcal{X}$-matroids intersects with several areas, including combinatorial rigidity and low-rank matrix completion. Notably, it is closely linked to the theory of maximal abstract rigidity matroids \cite{Somematroids,graver91rigidity,combrig} and maximal $H$-matroids \cite{maximum}.  

\medskip

The connection between our work and $\mathcal{X}$-matroids is highlighted in the following remark.

\begin{remark}
Let $\mathcal{X} \subset \textstyle \binom{[d]}{3}$. By applying a slightly modified version of Algorithm~\ref{algo m} to the initial hypergraph $\mathcal{X}$, with the condition that no two points within the same element of $\mathcal{X}$ are identified during the process, we obtain the set of all minimal $\mathcal{X}$-matroids of rank at most three.
\end{remark}

This leads to the following natural question.

\begin{question}
Let $\mathcal{X}\subset \textstyle \binom{[d]}{3}$. Find an algorithm to identify the set of minimal $\mathcal{X}$-matroids.
\end{question}

In \cite{jackson2024maximal}, the authors identified a sufficient condition under which the set of minimal $\mathcal{X}$-matroids contains a unique minimal element, and they conjectured that this condition is also necessary.
To state this precisely, we need the following definition from \cite{jackson2024maximal}.

\begin{definition}\normalfont
A {\em proper $\mathcal{X}$-sequence} is a sequence $\mathcal{S}=(X_{1},\ldots,X_{k})$ of sets in $\mathcal{X}$ such that $X_{i}\not \subset \cup_{j=1}^{i-1}X_{j}$ for each $i=2,\ldots,k$. For $F\subset [d]$, we define $\text{val}(F,\mathcal{S})=\lvert F\cup_{i=1}^{k} X_{i} \rvert -k$. The function $\text{val}_{\mathcal{X}}:2^{[d]}\rightarrow \mathbb{Z}$ is then given by
\[\text{val}_{\mathcal{X}}(F)=\min \{\text{val}(F,\mathcal{S}): \text{$\mathcal{S}$ is a proper $\mathcal{X}$-sequence}\}.\]
\end{definition}

In \cite{jackson2024maximal}, Jackson and Tanigawa proved the following lemma and conjectured its reverse.
\begin{lemma}\textup{\cite[Lemma~1.2]{jackson2024maximal}}\label{unique 23} %Let $\mathcal{X}$ be a collection of subsets of $[d]$. 
If $\textup{val}_{\mathcal{X}}$ is submodular and the set of $\mathcal{X}$-matroids is nonempty, then $\textup{val}_{\mathcal{X}}$ serves as the rank function of the unique minimal $\mathcal{X}$-matroid.
\end{lemma}

In the following example we show that the converse of this lemma does not hold. %holdspresent a counterexample to this conjecture.

\begin{example}\label{contra}
Let $\mathcal{X}=\{\{1,4,5\},\{2,4,5\},\{1,6,7\},\{2,6,7\},\{1,2,3\}\}$ on the ground set $[7]$. In this case, there is a unique minimal $\mathcal{X}$-matroid, which is the uniform matroid $U_{2,7}$. However, we have $\textup{val}_{\mathcal{X}}([7])=3$, which implies that $\textup{val}_{\mathcal{X}}$ is not the rank of the unique minimal $\mathcal{X}$-matroid.
\end{example}

We now explain the conceptual idea behind this example. Their conjecture %from \cite{jackson2024maximal} 
builds on \textup{\cite[Lemma~1.1]{jackson2024maximal}}, which provides an upper bound for the rank of any subset of the ground set in an $\mathcal{X}$-matroid by utilizing the submodularity of the rank function. In this lemma, submodularity is applied to recursively bound the rank of subsets that are the union of elements of $\mathcal{X}$. However, this application is restricted to situations where one of the two sets considered for the submodularity is an element of $\mathcal{X}$. It is possible, though, that this property could be extended to yield a tighter upper bound for the rank. This motivates the following definition.

\begin{definition}\label{defi}
We define a sequence of functions $\textup{val}^{n}_{\mathcal{X}}: 2^{[d]} \to \mathbb{Z}$ recursively, starting with $\textup{val}^{1}_{\mathcal{X}} = \textup{val}_{\mathcal{X}}$. For $n \geq 1$, each function $\textup{val}^{n+1}_{\mathcal{X}}$ is defined as follows:

\medskip
\begin{itemize}
\item[{\rm (i)}] If there exist subsets $A, B \subset [d]$ with $A \cap B \subset x$ for some $x \in \mathcal{X}$ such that
\begin{equation}\label{condition}
\textup{val}^{n}_{\mathcal{X}}(A\cup B)>\textup{val}^{n}_{\mathcal{X}}(A)+\textup{val}^{n}_{\mathcal{X}}(B)-\min \{\size{A\cap B},\size{x}-1\}.
\end{equation}
then %select subsets $A, B \subset [d]$ satisfying the inequality and 
define $\textup{val}^{n+1}_{\mathcal{X}}$ for any $F \subset [d]$ as:
\[
\textup{val}^{n+1}_{\mathcal{X}}(F) =
\begin{cases}
 \textup{val}^{n}_{\mathcal{X}}(F), & \text{if } F \neq A \cup B, \\
 \textup{val}^{n}_{\mathcal{X}}(A) + \textup{val}^{n}_{\mathcal{X}}(B) - \min \{\size{A \cap B}, \size{x} - 1\}, & \text{if } F = A \cup B.
\end{cases}
\]

\item[{\rm (ii)}] If condition (i) does not hold, but there exist subsets $A \subsetneq B \subset [d]$ such that
\begin{equation}\label{condition 2}
\textup{val}^{n}_{\mathcal{X}}(A)>\textup{val}^{n}_{\mathcal{X}}(B).
\end{equation}
then %select subsets $A, B \subset [d]$ as in (ii) and 
define $\textup{val}^{n+1}_{\mathcal{X}}$ for any $F \subset [d]$ as:
\[
\textup{val}^{n+1}_{\mathcal{X}}(F) =
\begin{cases}
 \textup{val}^{n}_{\mathcal{X}}(F), & \text{if } F \neq A, \\
 \textup{val}^{n}_{\mathcal{X}}(B), & \text{if } F = A.
\end{cases}
\]

\medskip
\item[{\rm (iii)}] If neither condition (i) nor (ii) holds, but there exist subsets $B \subsetneq A \subset [d]$ such that
\begin{equation}\label{condition 3}\textup{val}^{n}_{\mathcal{X}}(A)>\textup{val}^{n}_{\mathcal{X}}(B)+\size{A\setminus B}.
\end{equation}
then %select subsets $A, B \subset [d]$ as in (iii) and 
define $\textup{val}^{n+1}_{\mathcal{X}}$ for any $F \subset [d]$ as:
\[
\textup{val}^{n+1}_{\mathcal{X}}(F) =
\begin{cases}
 \textup{val}^{n}_{\mathcal{X}}(F), & \text{if } F \neq A, \\
 \textup{val}^{n}_{\mathcal{X}}(B) + \size{A \setminus B}, & \text{if } F = A.
\end{cases}
\]

\medskip
\item[{\rm (iv)}] If none of the above conditions hold, set $ \textup{val}^{n+1}_{\mathcal{X}} = \textup{val}^{n}_{\mathcal{X}}$.
\end{itemize}
There exists some $n$ such that $ \textup{val}^{n+1}_{\mathcal{X}} = \textup{val}^{n}_{\mathcal{X}}$. We then define $v_{\mathcal{X}}=\textup{val}^{n}_{\mathcal{X}}$.
\end{definition}

We now prove that the function $v_{\mathcal{X}}$ from Definition~\ref{defi} is well defined.

\begin{lemma}
The function $v_{\mathcal{X}}$ is independent of the choice of subsets at each step.
\end{lemma}

\begin{proof}
Let $v_{\mathcal{X}}$ be the function constructed in Definition~\ref{defi}, derived from the sequence $\textup{val}_{\mathcal{X}}^{n}$.
Suppose an alternative choice of subsets leads to another function $v_{\mathcal{X}}\rq$. We aim to show $v_{\mathcal{X}}\rq=v_{\mathcal{X}}$. To establish this equality, it suffices to prove $v_{\mathcal{X}}\rq\leq v_{\mathcal{X}}$, since the result follows by symmetry. 

We will show by induction that $v_{\mathcal{X}}\rq \leq \textup{val}_{\mathcal{X}}^{n}$. For the base case $n=1$, we have $\textup{val}_{\mathcal{X}}^{1} = \textup{val}_{\mathcal{X}}$, so the inequality holds trivially. For the inductive step, assume $v_{\mathcal{X}}\rq \leq \textup{val}_{\mathcal{X}}^{n}$ and analyze each case separately:

\medskip
{\bf Case~1.} Suppose there exist subsets $A,B\subset [d]$ with $A\cap B\subset x$ for some $x\in \mathcal{X}$ as in Equation~\eqref{condition}. By induction, we have:
\begin{equation*}
\begin{aligned}
\textup{val}_{\mathcal{X}}^{n+1}(A\cup B)&=\textup{val}_{\mathcal{X}}^{n}(A)+\textup{val}_{\mathcal{X}}^{n}(B)-\min \{\size{A\cap B},\size{x}-1\}\\
&\geq v_{\mathcal{X}}\rq(A)+v_{\mathcal{X}}\rq(B)-\min\{\size{A\cap B},\size{x}-1\}\\
&=v_{\mathcal{X}}\rq(A)+v_{\mathcal{X}}\rq(B)-v_{\mathcal{X}}\rq(A\cap B)\\
&\geq v_{\mathcal{X}}\rq(A\cup B),
\end{aligned}
\end{equation*}
where the last equality holds since $x\in \mathcal{X}$ and condition (i) does not hold for $v_{\mathcal{X}}\rq$. Thus, $v_{\mathcal{X}}\rq\leq \textup{val}_{\mathcal{X}}^{n+1}$.

\medskip
{\bf Case~2.} Suppose there exist subsets $A\subsetneq B\subset [d]$ as in Equation~\eqref{condition 2} and the first case does not apply. By induction, we have
\[\textup{val}_{\mathcal{X}}^{n+1}(A)=\textup{val}_{\mathcal{X}}^{n}(B)\geq v_{\mathcal{X}}\rq(B)\geq v_{\mathcal{X}}\rq(A),\]
where the last equality holds since condition (ii) does not hold for $v_{\mathcal{X}}\rq$. Thus, $\textup{val}_{\mathcal{X}}^{n+1}\geq v_{\mathcal{X}}\rq$.

\medskip
{\bf Case~3.} Suppose there exist subsets $B\subsetneq A\subset [d]$ as in Equation~\eqref{condition 3}, with neither Case~$1$ nor Case~$2$ occurring. By induction, we have
\[\textup{val}_{\mathcal{X}}^{n+1}(A)=\textup{val}_{\mathcal{X}}^{n}(B)+\size{A\setminus B}\geq v_{\mathcal{X}}\rq(B)+\size{A\setminus B}\geq v_{\mathcal{X}}\rq(A),\]
where the last equality holds since condition (iii) does not hold for $v_{\mathcal{X}}\rq$. Thus, $\textup{val}_{\mathcal{X}}^{n+1}\geq v_{\mathcal{X}}\rq$.

\medskip
{\bf Case~4.} If none of the previous cases apply, then $\textup{val}_{\mathcal{X}}^{n+1}=\textup{val}_{\mathcal{X}}^{n}$, and the result follows.
\end{proof}

The intuition behind this definition is that each $\textup{val}_{\mathcal{X}}^{n}$ serves as an upper bound for the rank of any $\mathcal{X}$-matroid. Each successive function refines this bound, becoming progressively sharper by leveraging the submodularity and monotonicity of the matroid rank function.

\begin{remark}
Observe that the second and the third conditions from the previous definition guarantee that, for any subset $A\subset [d]$ and any element $y\in [d]$,
\[\label{ineq}v_{\mathcal{X}}(A)\leq v_{\mathcal{X}}(A\cup \{y\})\leq v_{\mathcal{X}}(A)+1.\]
Therefore, if $v_{\mathcal{X}}$ is non-negative and submodular, it defines the rank function of a matroid.
\end{remark}

We now show that $v_{\mathcal{X}}$ provides an upper bound for the rank function of any $\mathcal{X}$-matroid.

\begin{lemma}\label{bound}
For any $\mathcal{X}$-matroid $M$ on $[d]$ and any $F\subset [d]$, we have $\rank_{M}(F)\leq v_{\mathcal{X}}(F)$. 
\end{lemma}

\begin{proof}
Let $v_{\mathcal{X}}$ be constructed from the sequence $(\textup{val}_{\mathcal{X}}^{n})_{n\in \mathbb{N}}$. We proceed by induction in $n$. For the base case, when $n=1$, the result holds because $\textup{val}_{\mathcal{X}}$ serves as an upper bound for $\rank_{M}$ by \textup{\cite[Lemma~1.1]{jackson2024maximal}}. For the inductive step, assume that %$\rank_{M}(F)\leq \textup{val}_{\mathcal{X}}^{n}(F)$ for all $F\subset [d]$. 
$\text{val}_{\mathcal{X}}^{n}$ is an upper bound for $\rank_{M}$. We analyze each case separately:

\medskip
{\bf Case~1.} Suppose there exist subsets $A,B\subset [d]$ with $A\cap B\subset x$ for some $x\in \mathcal{X}$ as in Equation~\eqref{condition}. By induction, we have:
\begin{equation*}
\begin{aligned}
\textup{val}_{\mathcal{X}}^{n+1}(A\cup B)&=\textup{val}_{\mathcal{X}}^{n}(A)+\textup{val}_{\mathcal{X}}^{n}(B)-\min \{\size{A\cap B},\size{x}-1\}\\
&\geq \rank_{M}(A)+\rank_{M}(B)-\min\{\size{A\cap B},\size{x}-1\}\\
&=\rank_{M}(A)+\rank_{M}(B)-\rank(A\cap B)\\
&\geq \rank_{M}(A\cup B),
\end{aligned}
\end{equation*}
where the last equality holds since $x\in \mathcal{C}(M)$. Thus, $\textup{val}_{\mathcal{X}}^{n+1}$ is an upper bound for $\rank_{M}$.

\medskip
{\bf Case~2.} Suppose there exist subsets $A\subsetneq B\subset [d]$ as in Equation~\eqref{condition 2} and the first case does not apply. By induction, we have
\[\textup{val}_{\mathcal{X}}^{n+1}(A)=\textup{val}_{\mathcal{X}}^{n}(B)\geq \rank_{M}(B)\geq \rank_{M}(A),\]
implying that $\textup{val}_{\mathcal{X}}^{n+1}$ is an upper bound for $\rank_{M}$.

\medskip
{\bf Case~3.} Suppose there exist subsets $B\subsetneq A\subset [d]$ as in Equation~\eqref{condition 3}, with neither Case~$1$ nor Case~$2$ occurring. By induction, we have
\[\textup{val}_{\mathcal{X}}^{n+1}(A)=\textup{val}_{\mathcal{X}}^{n}(B)+\size{A\setminus B}\geq \rank_{M}(B)+\size{A\setminus B}\geq \rank_{M}(A),\]
proving that $\textup{val}_{\mathcal{X}}^{n+1}$ is an upper bond for $\rank_{M}$.

\medskip
{\bf Case~4.} If none of the previous cases apply, then $\textup{val}_{\mathcal{X}}^{n+1}=\textup{val}_{\mathcal{X}}^{n}$, and the result follows. %immediately.
\end{proof}

Applying Lemma~\ref{bound} and following the argument in Lemma~\ref{unique 23}, we derive the following result.

\begin{lemma}\label{lema con}
%Let $\mathcal{X}$ be a collection of subsets of $[d]$. 
If $v_{\mathcal{X}}$ is submodular and the set of $\mathcal{X}$-matroids is nonempty, then $v_{\mathcal{X}}$ serves as the rank function of the unique minimal $\mathcal{X}$-matroid.
\end{lemma}

We conjecture that the converse of Lemma~\ref{lema con} is also true.

\begin{conjecture}
Suppose that the set of $\mathcal{X}$-matroids is nonempty. Then the function $v_{\mathcal{X}}$ is submodular if and only if there exists a unique minimal $\mathcal{X}$-matroid.
\end{conjecture}

\section{Identifying redundant matroid varieties in the decomposition}\label{appen}
Here, we present methods for identifying redundant matroid varieties within the decomposition of the circuit variety, %based on our algorithms, 
and provide proofs for the technical lemmas used in Section~\ref{examples}.

\subsection{Techniques for verifying redundancy}

In this subsection, we address the following central question and develop techniques to resolve it.

\begin{question}\label{question}
Given realizable matroids $M$ and $N$ of rank three on the same ground set, is $V_{N}\subset V_{M}$?
\end{question}

A necessary condition for an affirmative answer to Question~\ref{question} is provided by the following lemma.

\begin{lemma}
Let $M$ and $N$ be realizable matroids of rank three on $[d]$. %the same ground set. 
If $V_{N}\subset V_{M}$, then $N\geq M$.
\end{lemma}

\begin{proof}
Suppose, for contradiction, that $N\not \geq M$. Then there exists a subset $A\subset [d]$ such that $A\in \mathcal{D}(M)\backslash \mathcal{D}(N)$. For any realization $\gamma\in \Gamma_{N}$, the set $\{\gamma_{p}:p\in A\}$ is independent, contradicting the assumption that $\gamma\in V_{M}$.
\end{proof}

To explore Question~\ref{question}, we introduce the following definitions.

\begin{definition}\normalfont
%Let $M$ be a matroid of rank $n$ on the ground set $[d]$. 
The  projective realization space $\mathcal{R}(M)$ of a rank-three matroid $M$ consists of all the collections of points $\gamma=\{\gamma_{1},\ldots,\gamma_{d}\}\subset \mathbb{P}^{2}$ that satisfy the following condition:
\begin{equation*}
\{\gamma_{i_{1}},\ldots,\gamma_{i_{k}}\}\ \text{are linearly dependent} \Longleftrightarrow \{i_{1},\ldots,i_{k}\} \ \text{is dependent in $M$}.
\end{equation*}
The {\em moduli space} $\mathcal{M}(M)$ of the matroid $M$ is then defined as the quotient of this realization space by the action of the projective general linear group $\text{PGL}_{3}(\CC)=\text{GL}_{3}(\CC)/\CC^{\ast}$.
\end{definition}

Suppose $M$ contains a circuit of size $4$, which we can assume, up to relabeling, to be $\{1,2,3,4\}$. Each isomorphism class in $\mathcal{M}(M)$ has a unique representative $\gamma\in \mathcal{R}(M)$ where 
\begin{equation}\label{ecua matr}\{\gamma_{1},\gamma_{2},\gamma_{3},\gamma_{4}\}=
\begin{pmatrix}
1 & 0 &0 & 1\\
0 & 1 & 0 & 1\\
0 & 0 & 1 & 1
\end{pmatrix}.
\end{equation}
Thus, %the moduli space 
$\mathcal{M}(M)$ can be described as the set of all collections of points $\gamma=\{\gamma_{1},\ldots,\gamma_{d}\}\subset \mathbb{P}^{2}$ such that:
\begin{itemize}
\item Equality~\eqref{ecua matr} holds for $\{\gamma_{1},\ldots,\gamma_{4}\}$ and 
%\item %$\{\gamma_{i_{1}},\ldots,\gamma_{i_{k}}\}\ \text{are linearly dependent} \Longleftrightarrow \{i_{1},\ldots,i_{k}\} \ \text{is dependent in $M$}.$
$\gamma \in \mathcal{R}(M)$.
\end{itemize}

\begin{definition}
Let $M$ be a point-line configuration on $[d]$ with lines $\mathcal{L}$, and consider $l\in \mathcal{L}$ and $p\in l$. An {\em elementary perturbation} of $M$ at $\{l,p\}$ is the point-line configuration $\widetilde{M}$ on $[d]$ with the set of lines:
\[\mathcal{L}_{\widetilde{M}}=\begin{cases}
(\mathcal{L}\setminus \{l\})\cup \{l\setminus \{p\}\} \ & \text{if $\size{l}>3$,}\\
(\mathcal{L}\setminus \{l\}) \ & \text{if $\size{l}=3$.}
\end{cases}\]
An $m$-{\em perturbation} of $M$ is a sequence of $m$ elementary perturbations, 
$M_{0},M_{1},\ldots,M_{m}=M$,
where $M_{i-1}$ is an elementary perturbation of $M_{i}$ for each $i\in [m]$ and $M_{0}$ is solvable.
\end{definition}

The following result from \cite{guerville2023connectivity} offers a tool for addressing Question~\ref{question}. %, which will be applied frequently in the next section.

\begin{theorem}\textup{\cite[Theorem~4.5]{guerville2023connectivity}}\label{poli}
Let $M$ be a point-line configuration that admits an $m$-perturbation. %sequence $M_{0},M_{1},\ldots,M$. 
Then, there exists %a positive integer 
$d\in \mathbb{N}$, an open subset $U\subset \CC^{d}$ and $m$ polynomials $P_{1},\ldots,P_{m}$ such that
\[\mathcal{M}(M)\cong U\cap V(P_{1},\ldots,P_{m}).\]
\end{theorem}

Moreover, by following the constructive proof of this theorem in \cite{guerville2023connectivity} or by referring to \textup{\cite[Procedure~(1)]{Fatemeh6}}, we can obtain an explicit description of the open set $U$ and the polynomials $P_i$. %${1},\ldots,P_{m}$ from Theorem~\ref{poli}. 

\smallskip 

The following lemma offers an affirmative answer to Question~\ref{question} for a specific family of matroids.

\begin{lemma} \label{uti}
Let $M$ be a realizable point-line configuration on $[d]$ with lines $\mathcal{L}$, and let $l\in \mathcal{L}$ %be a line 
containing exactly three points. Let $N$ be the matroid of rank three with
\begin{equation*}
\mathcal{C}_{1}(N)=\emptyset, \quad \mathcal{C}_{2}(N)=\{\{i,j\}:\text{$i,j\in [d]\setminus l,i\neq j$}\}, \quad \text{and} \quad 
\mathcal{C}_{3}(N)=\{l\}.
\end{equation*}
Then, $V_{N}\subset V_{M}$.
\end{lemma}
\begin{proof}
Without loss of generality, assume that $l=\{1,2,3\}$. To prove the lemma, we will show that any collection of vectors $\gamma \in \Gamma_{N}$ can be infinitesimally perturbed to obtain $\widetilde{\gamma}\in \Gamma_{M}$. Since $N$ has a unique realization up to projective transformations, we can assume that
\[\gamma=\{\gamma_{1},\ldots,\gamma_{d}\}=
\begin{pmatrix}
1& 0 & 1 &0 &\ldots &\ldots & 0 \\
0&1 & 1 & 0& \ldots&\ldots & 0 \\
0& 0& 0& 1 &\ldots & \ldots & 1 
\end{pmatrix}.\]
Since %$M$ is realizable, 
$\Gamma_{M}$ is not empty, 
and since it is invariant under projective transformations, we know that there exists $\tau\in \Gamma_{M}$ of the form:
\[\tau=\{\tau_{1},\ldots,\tau_{d}\}=
\begin{pmatrix}
1& 0 & 1 &x_{1} &\ldots &\ldots & x_{d-3} \\
0&1 & 1 & y_{1}& \ldots&\ldots & y_{d-3} \\
0& 0& 0& 1 &\ldots & \ldots & 1 
\end{pmatrix}.\]
Moreover, for any $\lambda\in \CC\backslash \{0\}$, the collection of vectors 
\[\tau_{\lambda}=\{{\tau_{\lambda}}_{1},\ldots,{\tau_{\lambda}}_{d}\}=
\begin{pmatrix}
1& 0 & 1 &\lambda x_{1} &\ldots &\ldots & \lambda x_{d-3} \\
0&1 & 1 & \lambda y_{1}& \ldots&\ldots & \lambda y_{d-3} \\
0& 0& 0& 1 &\ldots & \ldots & 1 
\end{pmatrix},\]
also belongs to $\Gamma_{M}$. Therefore, for $\lambda$ infinitesimally close to zero, the collection of vectors $\tau_{\lambda}$ is an infinitesimal motion of $\gamma$, which completes the proof.
\end{proof}

\subsection{Proofs of lemmas from Section~\ref{examples}}

In this subsection, we provide proofs for the lemmas left unproven in Section~\ref{examples}.
\begin{proof}[{\bf Proof of Lemma~\ref{deco N}.}] By applying Algorithm~\ref{algo m}, we determine that $\min(N)$ consists of the following matroids; see Figure~\ref{min N} from left to right:

\begin{itemize}
\item[{\rm (i)}] The uniform matroid $U_{2,7}$.
\item[{\rm (ii)}] The matroids $N(i)$ for $i\in [7]$.
\item[{\rm (iii)}] The matroid obtained by identifying the points $\{1,4,6\}$, along with the matroids derived from it by applying an automorphism of $N$.
\item[{\rm (iv)}] The matroid obtained by identifying $1$ with $4$ and $2$ with $6$, along with the matroids derived from it by applying an automorphism of $N$.
\item[{\rm (v)}] The matroid obtained by identifying $1$ with $7$ and $2$ with $5$ together with the other two matroids derived from it by applying an automorphism of $N$.
\item[{\rm (vi)}] The matroid obtained by identifying $1$ with $5$ together with the other two matroids derived from it by applying an automorphism of $N$.
\end{itemize}

\begin{figure}[H]
    \centering
    \includegraphics[width=1\textwidth, trim=0 0 0 0, clip]{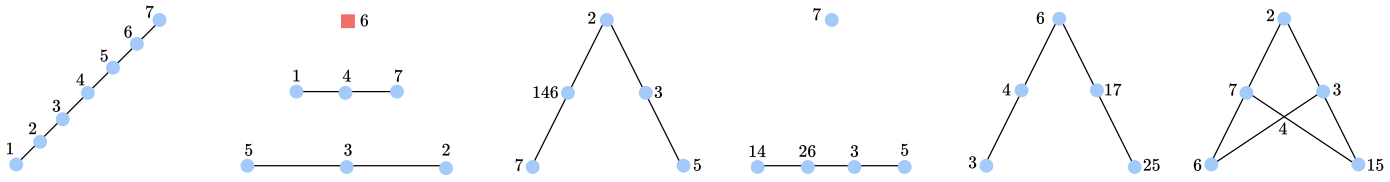}
    \caption{Minimal matroids of $N$~with~lines~$\{6,3,4\},\{6,2,7\},\{6,1,5\},\{2,3,5\},\{1,4,7\}$.}
    \label{min N}
\end{figure}

Let $A_{i}$ and $B_{i}$ for $i\in [6]$ denote the matroids of the third and fourth types, respectively. Similarly, let $C_{j}$ and $D_{j}$ for $j\in [3]$ denote the matroids of the fifth and sixth types. From this, we derive the following decomposition:
\begin{equation}\label{vn}
V_{\mathcal{C}(N)}=V_{N}\cup V_{U_{2,7}}\bigcup_{i=1}^{7}V_{\mathcal{C}(N(i))}\bigcup_{i=1}^{6}V_{\mathcal{C}(A_{i})}\bigcup_{i=1}^{6}V_{\mathcal{C}(B_{i})}\bigcup_{i=1}^{3}V_{\mathcal{C}(C_{i})}\bigcup_{i=1}^{3}V_{\mathcal{C}(D_{i})}.
\end{equation}
Note that the matroids $N(i),A_{i},B_{i}$ and $C_{i}$ are nilpotent, which implies that their matroid and circuit varieties coincide. Additionally, by Example~\ref{ej quad}, we have $V_{\mathcal{C}(D_{i})}=V_{D_{i}}\cup V_{U_{2,6}}$, since each $D_{i}$ is isomorphic to the quadrilateral set $QS$. Using these observations, Equation~\eqref{vn} becomes:
\begin{equation}\label{vn2}
V_{\mathcal{C}(N)}=V_{N}\cup V_{U_{2,7}}\bigcup_{i=1}^{7}V_{N(i)}\bigcup_{i=1}^{6}V_{A_{i}}\bigcup_{i=1}^{6}V_{B_{i}}\bigcup_{i=1}^{3}V_{C_{i}}\bigcup_{i=1}^{3}V_{D_{i}}.
\end{equation}

It is straightforward to verify that the matroid varieties $V_{A_{i}},V_{B_{i}},V_{C_{i}},V_{D_{i}}$ and $V_{N(j)}$ for $j\in \{1,2,3,4,5,7\}$ are contained within $V_{N}$. From this, we conclude that $V_{\mathcal{C}(N)}$ has the following irreducible decomposition:
$V_{\mathcal{C}(N)}=V_{N}\cup V_{U_{2,7}}\cup V_{N(6)}$,
which completes the proof.
\end{proof}

\medskip
\noindent
{\bf Proof of Lemmas~\ref{mov inc} and~\ref{mov inc 2}:} These results follow directly as corollaries of Lemma~\ref{uti}.

\begin{proof}[{\bf Proof of Lemma~\ref{dl}.}]
By applying Theorem~\ref{poli}, $\mathcal{M}(M_{\text{Pappus}})$ consists of all matrices of the form
\begin{equation}\label{matrs}\begin{pmatrix}
1 & 0 & 0 & 1 & 1 & y & 1 & x+y-1 & y \\
0 & 1 & 0 & 1 & x & y & 0 & xy & xy \\
0 & 0 & 1 & 1 & 0 & 1 & 1 & x & 1
\end{pmatrix},\end{equation}
where the columns correspond to the points $(1,2,8,9,3,7,6,5,4)$, respectively, and the minors corresponding to bases are nonzero.

(i) To establish this claim, we need to show that $\cup_{j=1}^{3} V_{\mathcal{C}(A_{j})}$ is contained in the right-hand side of Equation~\eqref{ecu dl}. Using Algorithm~\ref{algo m}, we identify the matroids $\{N: N\geq A_{j} \ \text{for some $j\in [3]$}\}$ that are not greater than or equal to any matroid associated with the varieties on the right-hand side of Equation~\eqref{ecu dl}. These matroids are obtained from  $M_{\text{Pappus}}$ by adding as lines some of the triples $\{1,4,9\},\{2,5,8\}$ and $\{3,6,7\}$.

To complete the proof, we will demonstrate that if $N$ is one such matroid, then $V_{N}\subset V_{M_{\text{Pappus}}}$. 
Specifically, given $\gamma \in \Gamma_{N}$, we will show that its vectors can be perturbed infinitesimally to yield $\widetilde{\gamma}\in \Gamma_{M_{\text{Pappus}}}$. Applying Theorem~\ref{poli}, it follows that $\gamma$ is of the form:
\[\begin{pmatrix}
1 & 0 & 0 & 1 & 1 & y & 1 & x+y-1 & y \\
0 & 1 & 0 & 1 & x & y & 0 & xy & xy \\
0 & 0 & 1 & 1 & 0 & 1 & 1 & x & 1
\end{pmatrix},\]
where the columns correspond to the points $(1,2,8,9,3,7,6,5,4)$, respectively, and the associated minors to bases are nonzero, except possibly for $\{1,4,9\},\{2,5,8\}$ and $\{3,6,7\}$. The vanishing of these minors is equivalent to the vanishing of the polynomials $xy-1,x+y-1$ and $xy+x-y$. By perturbing $x$ and $y$ infinitesimally to ensure these polynomials do not vanish, we obtain a collection $\widetilde{\gamma}\in \Gamma_{M_{\text{Pappus}}}$.

\medskip
(ii) Since all the matroids $B_{j}$ (Figure~\ref{min pap} (Center)) are related by automorphisms of $M_{\text{Pappus}}$, we may assume without loss of generality that we are considering the matroid $N$ obtained from $M_{\text{Pappus}}$ by identifying the points $\{3,4,5\}$. To establish this item, we need to show that $V_{\mathcal{C}(N)}\subset V_{M_{\text{Pappus}}}\cup V_{U_{2,9}}$.

By Theorem~\ref{nil coincide}, we know that $V_{\mathcal{C}(N)}=V_{N}\cup V_{U_{2,9}}$, so it remains to show that $V_{N}\subset V_{M_{\text{Pappus}}}$. To establish this, we show that for any $\gamma \in \Gamma_{N}$, we can perturb its vectors infinitesimally to obtain $\widetilde{\gamma}\in \Gamma_{M_{\text{Pappus}}}$. By applying a suitable projective transformation, we can assume that  $\gamma$ is of the form:
\[\begin{pmatrix}
1 & 0 & 0 & 1 & 1 & 1 & 1 & 1 & 1 \\
0 & 1 & 0 & 1 & z & 1 & 0 & z & z \\
0 & 0 & 1 & 1 & 0 & 0 & 1 & 0 & 0
\end{pmatrix},\]
where the columns correspond to the points $(1,2,8,9,3,7,6,5,4)$.
Setting $x=z$ and $y=\epsilon^{-1}$ in the matrix from Equation~\eqref{matrs}, we obtain:
\begin{equation}\label{matriz}\begin{pmatrix}
1 & 0 & 0 & 1 & 1 & 1 & 1 & 1+\epsilon(z-1) & 1 \\
0 & 1 & 0 & 1 & z & 1 & 0 & z & z \\
0 & 0 & 1 & 1 & 0 & \epsilon & 1 & \epsilon z & \epsilon 
\end{pmatrix},\end{equation}
which realizes $M_{\text{Pappus}}$ if the minors corresponding to bases do not vanish. By taking $\epsilon$ infinitesimally close to $0$, this matrix provides a realization $\widetilde{\gamma}\in \Gamma_{M_{\text{Pappus}}}$, giving an infinitesimal perturbation of $\gamma$.

\medskip
(iii) Since all the matroids $C_{k}$ (Figure~\ref{min pap} (Right)) are related by automorphisms of $M_{\text{Pappus}}$, we may assume without loss of generality that we are considering the matroid $E$ obtained from $M_{\text{Pappus}}$ by identifying $1=3,4=6$ and $7=9$. To establish this item, we need to show that $V_{\mathcal{C}(E)}\subset V_{M_{\text{Pappus}}}$. 

Since $E$ is nilpotent, we know that $V_{\mathcal{C}(E)}=V_{E}$, so it remains to prove that $V_{E}\subset V_{M_{\text{Pappus}}}$. To demonstrate this, we show that for any $\gamma \in \Gamma_{E}$, its vectors can be infinitesimally perturbed to obtain $\widetilde{\gamma}\in \Gamma_{M_{\text{Pappus}}}$. By applying a projective transformation, we may assume $\gamma$ is of the form:
\[\begin{pmatrix}
1 & 0 & 0 & 1 & 1 & 1 & 1 & z & 1 \\
0 & 1 & 0 & 1 & 0 & 1 & 0 & 1 & 0 \\
0 & 0 & 1 & 1 & 0 & 1 & 1 & 1 & 1
\end{pmatrix},\]
where the columns correspond to the points $(1,2,8,9,3,7,6,5,4)$.
By setting $y=(z-1)x+1$ in the matrix from Equation~\eqref{matrs}, we obtain the matrix
\begin{equation}\label{matri}\begin{pmatrix}
1 & 0 & 0 & 1 & 1 & (z-1)x+1 & 1 & z & (z-1)x+1 \\
0 & 1 & 0 & 1 & x & (z-1)x+1 & 0 & (z-1)x+1 & x((z-1)x+1) \\
0 & 0 & 1 & 1 & 0 & 1 & 1 & 1 & 1
\end{pmatrix},\end{equation}
which realizes $M_{\text{Pappus}}$ if the minors corresponding to bases do not vanish. By taking $x$ infinitesimally close to $0$, % in this matrix, 
we obtain a realization $\widetilde{\gamma} \in \Gamma_{M_{\text{Pappus}}}$, representing an infinitesimal perturbation of $\gamma$.
\end{proof}

\begin{proof}[{\bf Proof of Lemma~\ref{hl1}.}] (i) To prove this claim, we will show that $\cup_{l=1}^{18} V_{\mathcal{C}(H_{l})}$ is contained in the right-hand side of %Equation~
\eqref{equ papus 3}. Using Algorithm~\ref{algo m}, we identify the matroids $\{N: N\geq H_{l} \ \text{for some $j\in [18]$}\}$ that are not greater than or equal to any matroid associated with the varieties on the right-hand side of %Equation~
\eqref{equ papus 3}. These matroids, up to automorphism, are all at least one of the matroids obtained from  $M_{\text{Pappus}}$ by making $4$ a loop and adding as lines some of the triples $\{2,5,8\}$ and $\{3,6,7\}$.

To complete the proof, we will demonstrate that if $N$ is one such matroid, then $V_{N}\subset V_{M_{\text{Pappus}}(4)}$. 
Specifically, given $\gamma \in \Gamma_{N}$, we will show that its vectors can be perturbed infinitesimally to yield $\widetilde{\gamma}\in \Gamma_{M_{\text{Pappus}}(4)}$. Applying Theorem~\ref{poli}, it follows that $\gamma$ is of the form:
\[\begin{pmatrix}
1 & 0 & 0 & 1 & 1 & y & 1 & x+y-1 & 0 \\
0 & 1 & 0 & 1 & x & y & 0 & xy & 0 \\
0 & 0 & 1 & 1 & 0 & 1 & 1 & x & 0
\end{pmatrix},\]
where the columns correspond to the points $(1,2,8,9,3,7,6,5,4)$, respectively, and the minors corresponding to the bases do not vanish, except possibly for $\{2,5,8\}$ and $\{3,6,7\}$. The vanishing of this minors is equivalent to the vanishing of the polynomials $x+y-1$ and $xy+x-y$. Therefore, we can perturb the values of $x$ and $y$ infinitesimally in such a way that these equalities no longer hold, thus obtaining a collection $\widetilde{\gamma}\in \Gamma_{M_{\text{Pappus}}(4)}$, as desired.

\medskip
(ii) Since all the matroids $G_{k}$ are related by automorphisms of $M_{\text{Pappus}}$, we may assume without loss of generality that we are considering the matroid $N$ obtained from $M_{\text{Pappus}}$ by making the points $4$ and $5$ loops. To establish this item, we need to show that $V_{\mathcal{C}(N)}\subset V_{M_{\text{Pappus}}}$. 

Since $N$ is nilpotent, we know that $V_{\mathcal{C}(N)}=V_{N}$, 
so it remains to prove that $V_{N}\subset V_{M_{\text{Pappus}}}$. To demonstrate this, we show that for any $\gamma \in \Gamma_{N}$, its vectors can be infinitesimally perturbed to obtain $\widetilde{\gamma}\in \Gamma_{M_{\text{Pappus}}}$. By applying a projective transformation, we may assume $\gamma$ is of the form:
\begin{equation}\begin{pmatrix}\label{sñ}
1 & 0 & 0 & 1 & 1 & y & 1 & 0 & 0  \\
0 & 1 & 0 & 1 & x & y & 0 & 0 & 0 \\
0 & 0 & 1 & 1 & 0 & 1 & 1 & 0 & 0
\end{pmatrix},\end{equation}
where the columns correspond to the points $(1,2,8,9,3,7,6,5,4)$, respectively, and the minors corresponding to the bases do not vanish. By applying an infinitesimal perturbation to $\gamma$, we shift the last two columns from $(0,0,0)$ to $(x+y-1,xy,x)$ and $(x,xy,1)$, respectively. Furthermore, by making a small perturbation of $x$ and $y$ o ensure that the minors corresponding to the basis elements do not vanish, we obtain a realization of $M_{\text{Pappus}}$. This completes the proof.

\medskip
(iii) Without loss of generality, to prove this statement, it suffices to show that $V_{N}\subset V_{M_{\text{Pappus}}}$ for $N=M_{\text{Pappus}}(4)$. Let $\gamma \in \Gamma_{N}$.
By applying Theorem~\ref{poli}, we can assume that $\gamma$ is of the form:
\begin{equation}\begin{pmatrix}
1 & 0 & 0 & 1 & 1 & y & 1 & x+y-1 & 0  \\
0 & 1 & 0 & 1 & x & y & 0 & xy & 0 \\
0 & 0 & 1 & 1 & 0 & 1 & 1 & x & 0
\end{pmatrix},\end{equation}
where the columns correspond to the points $(1,2,8,9,3,7,6,5,4)$, respectively, and the minors corresponding to the bases do not vanish. By applying an infinitesimal perturbation to $\gamma$, we shift the last column from $(0,0,0)$ to $(x,xy,1)$. Additionally, by making a small adjustment to $x$ and $y$, we ensure that the minors corresponding to the basis elements do not vanish, thereby obtaining a realization of $M_{\text{Pappus}}$. This completes the proof.

\medskip
(iv) The matroids $F_{j}$ are isomorphic (up to the removal of loops) to the configuration of the three concurrent lines, so this item follows directly from \textup{\cite[Example~3.6]{clarke2021matroid}}.
\end{proof}

\begin{proof}[{\bf Proof of Lemma~\ref{l1}.}]\label{page:5.5.}
Applying Theorem~\ref{poli}, we find that $\mathcal{M}(\mathcal{K}_{9})$ consists of all matrices of form: 
\begin{equation}\label{matris}\begin{pmatrix}
1 & 0 & 0 & 1 & 1 & y & y-xy-1 & y-xy-1 & y\\
0 & 1 & 0 & 1 & x & y & -x & 0 & xy\\
0 & 0 & 1 & 1 & 0 & 1 & -x & -x & 1
\end{pmatrix},\end{equation}
where the columns represent the points $(1,6,2,7,3,8,9,4,5)$, respectively. Moreover, the minors corresponding to basis elements do not vanish and 
$\textstyle{\det \begin{pmatrix}
 1 & xy+1-y & y\\
 1 & 0 & xy\\
 1 & x & 1
\end{pmatrix}=0}.$

\medskip
(i) Since all the matroids $B_{i}$ are related by automorphisms of $\mathcal{K}_{9}$, we may assume without loss of generality that we are considering the matroid $N$ obtained from $\mathcal{K}_{9}$ by identifying the points $\{1,3,4,5,9\}$. To establish this item, we need to show that $V_{\mathcal{C}(N)}\subset V_{\mathcal{K}_{9}}$. 

Since $N$ is nilpotent, we know that $V_{\mathcal{C}(N)}=V_{N}$, 
so it remains to prove that $V_{N}\subset V_{\mathcal{K}_{9}}$. To demonstrate this, we show that for any $\gamma \in \Gamma_{N}$, its vectors can be infinitesimally perturbed to obtain $\widetilde{\gamma}\in \Gamma_{\mathcal{K}_{9}}$. By applying a projective transformation, we may assume $\gamma$ is of the form:
\[\begin{pmatrix}
1 & 0 & 0 & 1 & 1 & 1 & 1 & 1 & 1\\
0 & 1 & 0 & 1 & 0 & 1 & 0 & 0 & 0\\
0 & 0 & 1 & 1 & 0 & 0 & 0 & 0 & 0
\end{pmatrix},\]
where the columns correspond to the points $(1,6,2,7,3,8,9,4,5)$, respectively.  By performing the change of variables $y=e^{-1}$ in the matrix of Equation~\eqref{matris}, we obtain the matrix
\begin{equation}\label{matris 2}\begin{pmatrix}
1 & 0 & 0 & 1 & 1 & 1 & 1-x-\epsilon & 1-x-\epsilon & 1\\
0 & 1 & 0 & 1 & x & 1 & -x\epsilon & 0 & x\\
0 & 0 & 1 & 1 & 0 & \epsilon & -x\epsilon & -x\epsilon & \epsilon
\end{pmatrix},\end{equation}
which realizes $\mathcal{K}_{9}$ if the minors corresponding to basis elements do not vanish and the minor corresponding to  $\{4,5,7\}$ vanishes. By taking $x$ and $\epsilon$ infinitesimally close to $0$ in the matrix of Equation~\eqref{matris 2}, we obtain a realization $\widetilde{\gamma}\in \Gamma_{\mathcal{K}_{9}}$, representing an infinitesimal motion of $\gamma$.

\medskip
(ii) Since all the matroids $C_{i}$ are related by automorphisms of $\mathcal{K}_{9}$, we may assume without loss of generality that we are considering the matroid $N$ obtained from $\mathcal{K}_{9}$ by identifying the points $\{3,6\}$ and $\{2,4,5,8\}$. To establish this item, we need to show that $V_{\mathcal{C}(N)}\subset V_{\mathcal{K}_{9}}$. 

Since $N$ is nilpotent, we know that $V_{\mathcal{C}(N)}=V_{N}$, 
so it remains to prove that $V_{N}\subset V_{\mathcal{K}_{9}}$. To demonstrate this, we show that for any $\gamma \in \Gamma_{N}$, its vectors can be infinitesimally perturbed to obtain $\widetilde{\gamma}\in \Gamma_{\mathcal{K}_{9}}$. By applying a projective transformation, we may assume $\gamma$ is of the form:
\[\begin{pmatrix}
1 & 0 & 0 & 1 & 0 & 0 & 0 & 0 & 0\\
0 & 1 & 0 & 1 & 1 & 0 & 1 & 0 & 0\\
0 & 0 & 1 & 1 & 0 & 1 & 1 & 1 & 1
\end{pmatrix},\]
where the columns correspond to the points $(1,6,2,7,3,8,9,4,5)$, respectively.  Substituting $x=\epsilon^{-1}$ and $y=\epsilon t$ into the matrix of Equation~\eqref{matris}, we obtain the matrix
\begin{equation}\label{matris 3}\begin{pmatrix}
1 & 0 & 0 & 1 & \epsilon & \epsilon t & \epsilon t+\epsilon-\epsilon^{2}t & \epsilon t+\epsilon-\epsilon^{2}t & \epsilon t\\
0 & 1 & 0 & 1 & 1        & \epsilon t & 1 & 0 & t\\
0 & 0 & 1 & 1 & 0        & 1 & 1 & 1 & 1
\end{pmatrix},\end{equation}
which realizes $\mathcal{K}_{9}$ if the corresponding minors of basis elements do not vanish and the corresponding minor of the points $\{4,5,7\}$ vanishes. 
By letting $t$ and $\epsilon$ approach zero infinitesimally in the matrix of Equation~\eqref{matris 3}, we obtain a realization $\widetilde{\gamma}\in \Gamma_{\mathcal{K}_{9}}$, which represents an infinitesimal motion of $\gamma$.

\medskip
(iii) Since all the matroids $D_{i}$ are related by automorphisms of $\mathcal{K}_{9}$, we may assume without loss of generality that we are considering the matroid $N$ obtained from $\mathcal{K}_{9}$ by identifying the points $\{3,5,8\}$ and $\{2,4\}$. To establish this item, we need to show that $V_{\mathcal{C}(N)}\subset V_{\mathcal{K}_{9}}\cup V_{U_{2,9}}$.

From Theorem~\ref{nil coincide}, we know that $V_{\mathcal{C}(N)}=V_{N}\cup V_{U_{2,9}}$, so it remains to demonstrate that $V_{N}\subset V_{\mathcal{K}_{9}}$. To establish this, we show that for any $\gamma \in \Gamma_{N}$, we can perturb its vectors infinitesimally to obtain $\widetilde{\gamma}\in \Gamma_{\mathcal{K}_{9}}$. By applying a suitable projective transformation, we can assume that  $\gamma$ is of the form:
\[\begin{pmatrix}
1 & 0 & 0 & 1 & 1 & 1 & 0 & 0 & 1\\
0 & 1 & 0 & 1 & 1 & 1 & 1 & 0 & 1\\
0 & 0 & 1 & 1 & 0 & 0 & 1 & 1 & 0
\end{pmatrix},\]
where the columns correspond to the points $(1,6,2,7,3,8,9,4,5)$, respectively.  Substituting $y=\epsilon^{-1}$ into the matrix of Equation~\eqref{matris}, we obtain the matrix
\begin{equation}\label{matris 31}\begin{pmatrix}
1 & 0 & 0 & 1 & 1 & 1 & x+\epsilon-1 & x+\epsilon-1 & 1\\
0 & 1 & 0 & 1 & x        & 1 & x \epsilon & 0 & x\\
0 & 0 & 1 & 1 & 0        & \epsilon & x\epsilon & x\epsilon & \epsilon
\end{pmatrix},\end{equation}
which realizes $\mathcal{K}_{9}$ if the minors corresponding to basis elements do not vanish and the minor corresponding to the points  $\{4,5,7\}$ vanishes. 
By selecting $x$ and $\epsilon$ infinitesimally close to $1$ and $0$, respectively, in the matrix of Equation~\eqref{matris 31}, while ensuring that $(x+\epsilon-1)/x\epsilon$ is infinitesimally close to $0$, we obtain a realization $\widetilde{\gamma}\in \Gamma_{\mathcal{K}_{9}}$, which represents an infinitesimal motion of $\gamma$.

\medskip
(iv) Since all the matroids $E_{i}$ are related by automorphisms of $\mathcal{K}_{9}$, we may assume without loss of generality that we are considering the matroid $N$ obtained from $\mathcal{K}_{9}$ by identifying the points $\{7,8,9\}$. To establish this item, we need to show that $V_{\mathcal{C}(N)}\subset V_{\mathcal{K}_{9}}\cup V_{U_{2,9}}$.

From Theorem~\ref{nil coincide}, we know that $V_{\mathcal{C}(N)}=V_{N}\cup V_{U_{2,9}}$, so it remains to demonstrate that $V_{N}\subset V_{\mathcal{K}_{9}}$. To establish this, we show that for any $\gamma \in \Gamma_{N}$, we can perturb its vectors infinitesimally to obtain $\widetilde{\gamma}\in \Gamma_{\mathcal{K}_{9}}$. By applying a suitable projective transformation, we can assume that  $\gamma$ is of the form:
\[\begin{pmatrix}
1 & 0 & 0 & 1 & z & 1 & 1 & 1 & z\\
0 & 1 & 0 & 1 & 1 & 1 & 1 & 0 & 1\\
0 & 0 & 1 & 1 & 0 & 1 & 1 & 1 & z
\end{pmatrix},\]
where the columns correspond to the points $(1,6,2,7,3,8,9,4,5)$, respectively.  
By taking $y$ infinitesimally close to $1$ and $x$ infinitesimally close to $z^{-1}$ in the matrix of Equation~\eqref{sñ}, we obtain a realization $\widetilde{\gamma}\in \Gamma_{\mathcal{K}_{9}}$, which represents an infinitesimal motion of $\gamma$.
\end{proof}

The proof of Lemma~\ref{l5} follows easily from the same arguments used in the previous lemmas.

\vspace{-3mm}

\bibliographystyle{plain}  % Or another style like 'amsalpha', 'unsrt', etc.
\bibliography{Citation}

%\printbibliography

\medskip
\noindent 
\small{\textbf{Authors' addresses}

\noindent
Department of Mathematics, KU Leuven, %Celestijnenlaan 200B, B-3001 
Leuven, Belgium
\\ E-mail address: {\tt emiliano.liwski@kuleuven.be}
\\
   E-mail address: {\tt fatemeh.mohammadi@kuleuven.be}

\end{document}